\documentclass[12pt,leqno]{amsart}\usepackage[latin9]{inputenc}\usepackage{verbatim,amsthm,amstext,amssymb,color,amscd,bm}\makeatletter
\usepackage{hyperref}\hypersetup{colorlinks=true,citecolor=black,linkcolor=black}\usepackage[shortlabels]{enumitem}\numberwithin{equation}{section}
\theoremstyle{plain}\newtheorem{theorem}{Theorem}[section]\newtheorem{lemma}[theorem]{Lemma}
\newtheorem{corollary}[theorem]{Corollary}\theoremstyle{definition}\newtheorem{remark}[theorem]{Remark}\DeclareMathOperator{\mdim}{\mathrm{mdim}}
\DeclareMathOperator{\Widim}{\mathrm{Widim}}\DeclareMathOperator{\Map}{\mathrm{Map}}\DeclareMathOperator{\Sym}{\mathrm{Sym}}\makeatother
\begin{document}
\title[Sofic approximation sequences and sofic mean dimension]{Sofic approximation sequences and\\sofic mean dimension}
\author[Lei Jin]{Lei Jin}
\address{Lei Jin: School of Mathematics, Sun Yat-sen University, Guangzhou, China}
\email{jinleim@mail.ustc.edu.cn}
\author[Yixiao Qiao]{Yixiao Qiao}
\address{Yixiao Qiao (Corresponding author): School of Mathematics and Statistics, Guangdong University of Technology, Guangzhou, China}
\email{yxqiao@mail.ustc.edu.cn}
\begin{abstract}
The main purpose of this paper is to strengthen our understanding of sofic mean dimension of two typical classes of sofic group actions. First, we study finite group actions. We prove that sofic mean dimension of any amenable group action does not depend on the choice of sofic approximation sequences. Previously, this result was known only if the acting group is an infinite amenable group. Moreover, we investigate the full shifts, for all sofic groups and all alphabets. We show that sofic mean dimension of any full shift depends purely on its alphabet. Our method is a refinement of the classical technique in relation to the estimates from above and below, respectively, for mean dimension of some typical actions. The key point of our results is that they apply to all compact metrizable spaces without any restriction (in particular, the alphabet concerned in a full shift and the space involved in a finite group action are not required to be finite-dimensional).

Furthermore, we improve the quantitative knowledge of sofic mean dimension, restricted to finite-dimensional compact metrizable spaces, for those two typical classes of sofic group actions. As a direct consequence of the main ingredient of our proof, we obtain the exact value of sofic mean dimension of all the actions of finite groups on finite-dimensional compact metrizable spaces. Previously, only an upper bound for these actions was given. Besides, we also get the exact value of sofic mean dimension of full shifts when the alphabet is finite-dimensional.
\end{abstract}
\maketitle

\medskip

\section{Introduction}
\subsection{Main results}
Mean dimension is a topological invariant of dynamical systems. It originates with Misha Gromov \cite{Gromov2} in 1999 and was investigated systematically by Elon Lindenstrauss and Benjamin Weiss \cite{LW} around 2000 within the framework of amenable group actions. Mean dimension has now proved to achieve remarkable success in the study of dynamical systems. To make it reach to a much broader class of group actions, Hanfeng Li \cite{Li} introduced the notion of \textit{sofic mean dimension} from a fairly new perspective in 2013. This is an excellent generalization of the classical version of mean dimension, which enables it to apply to all sofic group actions. The sofic groups, that appeared initially with the article \cite{Gromov1} by Gromov in 1999, do form a rather extensive family of groups, and contain in particular all the amenable groups and all the residually finite groups.\footnote{Actually, it has not yet been verified if there exists a countable group which is not sofic.} Along with the definition of sofic mean dimension, Li \cite{Li} successfully built its connection, as we expected, with the classical version of mean dimension from his modern point of view.

To state Li's result precisely, we put necessary conventions here very briefly. The detailed description of our terminology can be found in Section 2. Throughout this paper, all the acting groups are always assumed to be \textit{countable} (i.e., either finite or countably infinite) and \textit{discrete}. If a sofic group $G$ acts continuously on a compact metrizable space $X$ and if $\Sigma$ is a sofic approximation sequence for $G$, then its sofic mean dimension is denoted by $\mdim_\Sigma(X,G)$. Formally, it is called the sofic mean dimension of $(X,G)$ with respect to the sofic approximation sequence $\Sigma$ for the acting group $G$. We notice that the notation for mean dimension (i.e., the classical version of mean dimension) of amenable group actions is not involved in this paper.

Li \cite[Section 3]{Li} showed that if an \textit{infinite} amenable group $G$ acts continuously on a compact metrizable space $X$ and if $\Sigma$ is a sofic approximation sequence for $G$, then $\mdim_\Sigma(X,G)$ is equal to its mean dimension. This is a highly satisfactory bridge between these two levels. However, a finite group action may be\footnote{Note that automatically, any finite group is amenable.} an exception. In fact, for \textit{double finite actions} $(X,G)$ (namely, a \textit{finite} group $G$ acts continuously on a \textit{finite-dimensional} compact metrizable space $X$), Li \cite[Section 3]{Li} provided an upper bound for $\mdim_\Sigma(X,G)$, which is finer than the mean dimension of $(X,G)$ and which leads immediately to a bunch of examples of $(X,G)$ such that the sofic mean dimension $\mdim_\Sigma(X,G)$ and the mean dimension of $(X,G)$ do not coincide.

We remark here that the sofic mean dimension of any double finite action remained unclear at that time. Recently, we finally settled this question in a previous version of the present paper, obtaining the \textit{exact} value of sofic mean dimension for all the double finite actions (as specified now in Subsection 1.3 of the current version). Nowadays we prefer to treat it as a corollary of the main proposition of this paper rather than to publish our original proof with that early manuscript elsewhere.

Moreover, as a direct consequence, this answer to the above question gives an ``if and only if'' condition under which such a pleasant equality (connecting those two levels) for double finite actions becomes true or false.

Even so, the known results in this direction (e.g. sofic mean dimension of finite group actions) still do not depart from the usual assumption that the space is additionally required to be finite-dimensional. Indeed, as we will see in a moment, general and sharp statements about sofic mean dimension of full shifts are also restricted to finite-dimensional alphabets (i.e., compact metrizable spaces) only. There is a lack of knowledge of such assertions for \textit{infinite-dimensional} compact metrizable spaces. The reason behind this restriction will be explained at length in subsequent subsections.

Nonetheless, let us turn to a natural issue which is unknown by reason of the same obstacle, too, but which we have been able to address. Although it seems somewhat unfortunate that we \textit{cannot} always expect an amenable group action to have sofic mean dimension agreeing with its mean dimension, a further question that is worth studying is to find some slightly weaker property that could encompass \textit{all} the amenable group actions. As mentioned previously, actions of infinite amenable groups and double finite actions have sofic mean dimension already very clear to us. In particular, this implies that the value $\mdim_\Sigma(X,G)$, for any action $(X,G)$ among them, is independent of the sofic approximation sequences $\Sigma$ for $G$. Nevertheless, this is not confirmed when the acting group $G$ is finite and the space $X$ is infinite-dimensional. We solve this problem. It is quite reasonable to expect all the amenable group actions to possess such a more essential (and more abstract) property. The first main result of this paper is to establish this statement.

\begin{theorem}[Main theorem 1]\label{main1}
Sofic mean dimension of any amenable group action does not depend on the choice of sofic approximation sequences.
\end{theorem}

More precisely, we shall prove an inner equality (in relation to sofic mean dimension) for the class of amenable group actions:
\begin{itemize}\item
If an amenable group $G$ acts continuously on a compact metrizable space $X$, and if $\Sigma$ and $\Sigma^\prime$ are two sofic approximation sequences for $G$, then we have $\mdim_\Sigma(X,G)=\mdim_{\Sigma^\prime}(X,G)$.
\end{itemize}
Instead of an outer equality between the defined values at those two levels, Theorem \ref{main1} eventually allows us to unify the reduction of sofic mean dimension to mean dimension with a view towards a common value shared among all the approximation sequences for an acting group (i.e. uniquely determined\footnote{We would like to remind the reader to keep in mind that this makes it reasonable to remove the approximation sequence from the notation of mean dimension, while it is not planned to mean if this value is definable or not in the context of the non-existence of an approximation sequence.} by the action).

\medskip

As mentioned above, the main new ingredient of Theorem \ref{main1} is the statement for finite group actions which are considered, together with the full shifts, as the most standard method for generating a group action from a space, in the sense that they reflect the topological nature of a space with some canonical (sometimes, even trivial) dynamical behaviours. Conversely, these typical actions have some desired feature similar to the phenomena in (topological) dimension theory, for example, in the universality aspect. In particular, mean dimension is closely related to the (dynamical) \textit{embedding problem} (i.e., embedding dynamical systems into the shift action on Hilbert cubes) which is a wonderful application of mean dimension to dynamical systems. We shall not study this topic in this paper. Therefore we do not describe in detail how it deeply relates (abstract) dynamical systems to different areas (e.g. classic analysis). For the latest progress on this problem we refer the reader to \cite{LW,LT,Gutman,GTshifts,GT,GQT,JQhilbert}. From all those facts and results we became aware that any careful understanding of mean dimension of full shifts is valuable.

Let $G$ be a group. Let $K$ be a compact metrizable space. We denote by $(K^G,\sigma_G)$ the shift action of $G$ on the product space $K^G$. Usually it is simply called a full shift over the alphabet $K$.

The alphabet $K$ plays a crucial role in the full shift $(K^G,\sigma_G)$. Lindenstrauss and Weiss \cite{LW} showed that when $G$ is an amenable group and $K=[0,1]^D$ (where $D$ is a positive integer, or possibly, $+\infty$), the mean dimension of $(K^G,\sigma_G)$ is equal to $D$, which is the same as the dimension (by which, here and in the sequel, we mean the topological dimension, namely, the Lebesgue covering dimension) of $K$. Since mean dimension theory is apparently an analogue of dimension theory, this result naturally brings about a seemingly plausible impression, i.e., the mean dimension of the full shift $(K^G,\sigma_G)$ is equal to the dimension of the alphabet $K$ on all occasions. However, this turns out to be incorrect in general. Masaki Tsukamoto \cite{Tsukamoto} proved a satisfactory result which surprisingly denies this impression and which enables mean dimension of full shifts over \textit{finite-dimensional} alphabets to be understood completely. As far as we noted somewhere else, although Tsukamoto's result \cite{Tsukamoto} is stated for the case $G=\mathbb{Z}$, generalising it to all the amenable groups $G$ is very straightforward. For instance, this can be fulfilled, without any additional effort, with the help of \cite{JQ}.

Before we proceed any further, we would like to remark that in the context of amenable groups the key difficulty with the mean dimension of full shifts is an effective estimate from below (which was conquered by Tsukamoto \cite{Tsukamoto}). This is because the mean dimension of a full shift is dominated from above by the dimension of its alphabet \cite{LW}, which can improve itself (e.g. employing \cite{JQ}) up to an optimal upper bound (provided that the alphabet is finite-dimensional) with a standard trick. This procedure however does not apply to the sofic framework any more (for details, please refer to Subsection 3.4), which consequently becomes a main obstacle (to the exact value of sofic mean dimension of full shifts) to overcome.

For a sofic group $G$ and a sofic approximation sequence $\Sigma$ for $G$, Li \cite{Li} showed that $\mdim_\Sigma(K^G,\sigma_G)$ is bounded from above by the dimension of the alphabet $K$ (which is generally not optimal). In an earlier (unpublished) version of this paper, the authors (in an effort to deal with sofic approximation sequences for acting groups) refined Li's estimate with a substantially different method, and thus, successfully extended Tsukamoto's result \cite{Tsukamoto} to all the sofic groups $G$ as long as the alphabet $K$ is finite-dimensional. Nowadays this statement is specified within Subsection 1.3 of the present paper (as it is currently processed with a corollary of the main proposition of the present paper).

It is important to notice that among all the above results, the additional condition that assumes the alphabet $K$ to be finite-dimensional is essential to the proof. By reason of almost the same obstruction, it is not clear to us for a long time if the value $\mdim_\Sigma(K^G,\sigma_G)$, for an arbitrary alphabet $K$, will change as we choose different sofic approximation sequences $\Sigma$, or even along with different sofic groups $G$. The previous result carried out by the authors (in the early version of the present paper, as mentioned above) implies in particular that when the alphabet $K$ is finite-dimensional, the term $\mdim_\Sigma(K^G,\sigma_G)$ does be independent of the sofic approximation sequences $\Sigma$ for the group $G$. The purpose of our second main result is to carry out this assertion for all alphabets, proving that sofic mean dimension of any full shift depends purely on its alphabet.

We recall our general setting in front of the statement of the second main result. Let $K$ be a compact metrizable space. Let $G$ and $G^\prime$ be sofic groups. Let $\Sigma$ and $\Sigma^\prime$ be sofic approximation sequences for $G$ and $G^\prime$, respectively. We have an equality for sofic mean dimension of full shifts all the time:
\begin{theorem}[Main theorem 2]\label{main2}
$$\mdim_\Sigma(K^G,\sigma_G)=\mdim_{\Sigma^\prime}(K^{G^\prime},\sigma_{G^\prime}).$$
\end{theorem}
The point of Theorem \ref{main2} is that the statement applies, without any restriction, to all the full shifts of all sofic groups along with any of their sofic approximation sequences. In particular, the alphabet $K$ in this equality is not required to be finite-dimensional any longer.

\medskip

\subsection{Strategy}
This subsection aims to explain\footnote{For convenience, we shall not get into technical details in this subsection. But we believe that an experienced reader interested in this topic may see the point of our strategy.} the difficulty with our main results and our main ideas. Our main purpose, as stated above, is to achieve Theorem \ref{main1} eventually. Since the proof of Theorem \ref{main1} has some similarity, especially in the key estimates, to the proof of Theorem \ref{main2}, for simplicity we only describe how we reach Theorem \ref{main2}. The problem caused by (the soficity of) the acting groups (i.e., going from amenable groups to sofic groups) will be explained in detail and settled in the sequel. After we address the soficity of the acting groups, we will finally accomplish all our main results. So, in the current subsection, let us focus mainly on the difficulty arising from the spaces rather than the groups.

As mentioned before, in an effort to work on the sofic approximation sequences for the acting groups, we have already been able to get a clear picture of these results for all the finite-dimensional compact metrizable spaces in an early version of this paper (as specified now in Subsection 1.3 of the present version). It then suffices to concentrate on processing the statement for infinite-dimensional ones. A previous problem in dealing with them is to produce an effective lower bound for sofic mean dimension. The following citations concern such an estimate from below, which are specifically related to the approach given in \cite{Tsukamoto}, which may fail unfortunately, and seems impossible to be improved with the same method, if the alphabet $K$ is infinite-dimensional.

More precisely, it was first noted by Tsukamoto \cite{Tsukamoto} that at least two difficulties appeared when handling the full shifts over infinite-dimensional alphabets, which come from two remarkable phenomena in infinite-dimensional topology. Quoting \cite{Tsukamoto} (with a slight modification):
\begin{itemize}
\item\textit{There exists an infinite-dimensional compact metrizable space $K$ containing no intermediate dimensional subspaces. Namely, every closed subset $A$ of $K$ has dimension either $0$ or $\infty$.}
\item\textit{There exists an infinite-dimensional compact metrizable space $K$ which cohomologically looks like a surface. Namely, every closed subset $A$ of $K$ satisfies that for any $n\ge3$ the \v{C}ech cohomology group $\check{H}^n(K,A)$ vanishes.}
\end{itemize}
Furthermore, Tsukamoto \cite{Tsukamoto} posed the precise reason (from outer and inner aspects, respectively) why they appear to become essential obstacles to the issue. Quoting \cite{Tsukamoto} (with a slight modification):
\begin{itemize}
\item\textit{These two difficulties are genuinely infinite-dimensional phenomena. The first difficulty implies that we cannot reduce the problem to a finite-dimensional case, and the second difficulty implies that the ordinary cohomology theory is insufficient to solve the problem.}
\end{itemize}

We overcome\footnote{Strictly speaking, we avoid them.} the above difficulties by taking a detour. In comparison with the previous work, we no longer pay attention to the exact value of sofic mean dimension of full shifts. As an alternative, our idea turns to looking for a common expression that can be shared between the estimates from both above and below. Thus, to this attempt our first problem is as follows:
\begin{itemize}\item
How to find a good candidate which is able not only to unify the upper and lower bounds qualitatively for sofic mean dimension of full shifts, but also to remove the group information (e.g. about the sofic approximation sequences for an acting group) quantitatively from the expression of estimates?
\end{itemize}
The word ``good'' here is just planned to mean that making such a candidate work on all the alphabets will be within the authors' reach.

We observe that the candidate selected in \cite{Tsukamoto} for this purpose is the term of dimension, which is actually one of the most fundamental and somewhat global invariants of topological spaces. This definitely meets all the expectations, but meanwhile, it exactly causes the difficulty (which lies mainly in one direction of the estimates).

We shall adopt a different method for realising it. Instead of the dimension of a compact metrizable space, we develop more delicate estimates for our aim (i.e. for sofic mean dimension of these typical actions) from both of the two directions, which can adapt directly to each other, with terms of $\epsilon$-width dimension. This is milder in some sense. Intuitively speaking, we take a step backwards, and in consequence, we have sight less clear than before; however, it then becomes possible for us to get a wider range of vision, such that we are almost able to treat finite-dimensional and infinite-dimensional alphabets with a unified process.

Although the candidate that we choose has some advantage, we have to be aware that it has correspondingly weaker properties. We do not list them here, as they will be settled with a more careful technique in the proof. The method that we provide is a refinement of the classical method for producing estimates for sofic mean dimension from above and below, respectively. In contrast to the classical technique, our approach to the estimates applies to both finite-dimensional and infinite-dimensional compact metrizable spaces, and in particular, it does \textit{not} rely on a practical lemma\footnote{This lemma was widely used for a similar purpose (i.e., if there is an obstacle heavily dependent on an effective estimate for $\epsilon$-width dimension from below) by the researchers interested in those topics relevant to this direction (e.g., we employ it to realize Theorem \ref{main3} and Theorem \ref{main4}).} for $\epsilon$-width dimension explored by Tsukamoto \cite{Tsukamoto}. This is the major difference between our method and the previous ones.

\medskip

\subsection{Quantitative main results}
As stated above, for any amenable group action $(X,G)$ Li \cite[Section 3]{Li} showed (with respect to any sofic approximation sequence $\Sigma$ for $G$) that $\mdim_\Sigma(X,G)$ is bounded from above by the mean dimension of $(X,G)$; if in addition, $G$ is infinite, then these two values coincide. However, in connection with a finite group $G$ this equality may be false, namely, sofic mean dimension is possible to be different from (i.e. strictly less than the classical version of) mean dimension.

More precisely, we let a finite group $G$ act continuously on a finite-dimensional compact metrizable space $X$. Note that in this case, it follows directly from the definition that the mean dimension of $(X,G)$ is equal to $\dim(X)/|G|$, an explicitly known value. Nevertheless, in contrast to the case that the acting group is infinite (and amenable), the exact value of its sofic mean dimension (with respect to a sofic approximation sequence $\Sigma$ for $G$) is not clear to us. In fact, Li \cite[Section 3]{Li} proved an upper bound for $\mdim_\Sigma(X,G)$, which is finer than $\dim(X)/|G|$, as follows: $$\mdim_\Sigma(X,G)\le\frac{1}{|G|}\cdot\inf_{n\in\mathbb{N}}\frac{\dim(X^n)}{n}.$$ The third main result of this paper is to strengthen Li's estimate with an equality which then gives the exact value of sofic mean dimension for all the actions of finite groups on finite-dimensional compact metrizable spaces.

\begin{theorem}[Main theorem 3]\label{main3}
If a finite group $G$ acts continuously on a finite-dimensional compact metrizable space $X$, and if $\Sigma$ is a sofic approximation sequence for $G$, then $$\mdim_\Sigma(X,G)=\frac{1}{|G|}\cdot\inf_{n\in\mathbb{N}}\frac{\dim(X^n)}{n}.$$
\end{theorem}

\medskip

Parallel to sofic mean dimension of finite group actions, here our attention naturally turns to sofic mean dimension of full shifts. As described just now in Subsection 1.1 about recent progress on mean dimension of full shifts, for the group $\mathbb{Z}$ the exact value was obtained satisfactorily by Tsukamoto \cite{Tsukamoto} provided the alphabet $K$ is finite-dimensional, and also noticed that by means of \cite{JQ} it is straightforward to generalize this result to amenable groups with almost no additional effort. But it turns out that the situation (as explained shortly in Subsection 3.4) becomes essentially different if we proceed to the case of sofic groups.

Actually, Li \cite[Section 7]{Li} presented (for any sofic approximation sequence $\Sigma$ for a sofic group $G$) an upper bound: $$\mdim_\Sigma(K^G,\sigma_G)\le\dim(K).$$ In general, this upper bound for $\mdim_\Sigma(K^G,\sigma_G)$ however is not optimal. The key ingredient of our fourth main result is to refine Li's estimate (with a method substantially different from Li's\footnote{Roughly speaking, Li's approach can be regarded as (in some sense, a kind of) local analysis, i.e. using the notion of \textit{sofic metric mean dimension} which dominates sofic mean dimension from above. On the contrary, we shall adopt a global treatment for $\mdim_\Sigma(K^G,\sigma_G)$.}), so as to get the exact value of sofic mean dimension of full shifts over any finite-dimensional alphabet.

\begin{theorem}[Main theorem 4]\label{main4}
Let $K$ be a finite-dimensional compact metrizable space. Let $G$ be a sofic group and $\Sigma$ a sofic approximation sequence for $G$. The following assertion is true:
$$\mdim_\Sigma(K^G,\sigma_G)=\inf_{n\in\mathbb{N}}\frac{\dim(K^n)}{n}.$$
\end{theorem}

\medskip

\subsection{Further discussion on open problems}
The aim of this subsection is to pose two open problems in company with our remark, which seem to be worth considering. In the direction of the previously obtained results, these questions appear to be natural and fundamental, and as a consequence, solutions to them will give rise to a complete picture of sofic mean dimension of amenable group actions and full shifts. We put them as follows. As usual let $\Sigma$ be a sofic approximation sequence for an acting group $G$.
\begin{itemize}
\item Let $K$ be an infinite-dimensional compact metrizable space. Is the statement $\mdim_\Sigma(K^G,\sigma_G)=+\infty$ true?
\item Let a finite group $G$ act continuously on an infinite-dimensional compact metrizable space $X$. Is the statement $\mdim_\Sigma(X,G)=+\infty$ true?
\end{itemize}
The first problem arose originally from \cite{Tsukamoto} where it was stated for the case $G=\mathbb{Z}$, while the second problem was raised initially by the authors in an early version of this paper.

Also, as indicated in the previous subsections, these two problems are intimately connected with each other as well as they are involved in our main results. The novel ingredient here is to point out (a little further on) that they reduce to each other (i.e., they are actually \textit{equivalent}):
\begin{itemize}\item
Let $K$ be a compact metrizable space. Let $G$ be a sofic group and $\Sigma$ a sofic approximation sequence for $G$. Let $G^\prime$ be a finite group and $\Sigma^\prime$ a sofic approximation sequence for $G^\prime$. Let $G^\prime$ act continuously on $K$. We have that $\mdim_\Sigma(K^G,\sigma_G)=+\infty$ if and only if $\mdim_{\Sigma^\prime}(K,G^\prime)=+\infty$.
\end{itemize}
This fact follows directly from the main proposition of this paper, which is located in Section 3.

\medskip

\subsection{Organization of the paper}
By the end of this section we overview the structure of the present paper. In Section 2, we briefly review basic definitions related to sofic mean dimension. In Section 3, we state the main proposition of this paper. Moreover, as several of its corollaries, we prove our main theorems (i.e., Theorem \ref{main1} and Theorem \ref{main2}) and quantitative main theorems (i.e., Theorem \ref{main3} and Theorem \ref{main4}). In addition, we include a toy-model of the main proposition (in company with a simple proof), which is logically independent of all the results mentioned in the paper. This part is dedicated to the technical intersection of difficulties arising from the spaces and the acting groups; namely, it explains (for this simple case) how we remove the finite-dimensional restriction from the spaces, and also, it helps us to see the additional obstacles (in a more precise way as we would show) to going a step forward (from amenable groups to sofic groups). Section 4 and Section 5 are then produced, devoted to the proof of our main proposition. Besides, two appendices are included. Appendix A is logically independent of the body of this paper, and it contains, for the reader's convenience, a standard explanation about two different approaches to the definition of sofic mean dimension. Appendix B consists of a dichotomy in dimension theory together with some corollaries and remarks (assuming the quantitative main results).

\medskip

\subsection*{Acknowledgements}
This paper was written when the first-named author stayed in for a time during his depression. He would like to take this opportunity to thank his wife for keeping persistent company with him through the difficult period. L. Jin was supported by NNSF of China No. 12201653. Y. Qiao was supported by NNSF of China No. 12371190.

\medskip

\section{Review of sofic mean dimension}
\subsection{Background}
The background of sofic mean dimension has been contained briefly in Section 1 already, which proved to be the main motivation of this paper. So we do not repeat it here. We list fundamental material on terminologies and notations. The writing of this section is mostly borrowed, for consistency, from the paper \cite{JQ} by the authors, with some necessary modification adapting to the present paper. We are now starting with basic definitions. Throughout this paper the symbol $\mathbb{N}$ is to denote the set of positive integers.

\subsection{Group actions}
Let $G$ be a group. By the terminology ``$G$ \textbf{acts continuously on} a compact metrizable space $X$'' we understand a continuous mapping $$\Phi:G\times X\to X,\quad(g,x)\mapsto gx$$ satisfying the following conditions: $$\Phi(e,x)=x,\quad\Phi(gh,x)=\Phi(g,\Phi(h,x)),\quad\forall x\in X,\;\forall g,h\in G,$$ where $e$ is the identity element of the group $G$. We generally omit the mapping $\Phi$ provided we have already gotten the action of $G$ on $X$ clear.

The full shifts are a typical class of group actions. This notion was concerned in Theorem \ref{main2}. Let $K$ be a compact metrizable space. The \textbf{shift action} of $G$ on the product space $K^G$ is defined as follows: $$\sigma_G:G\times K^G\to K^G,\quad(g,(x_h)_{h\in G})\mapsto(x_{hg})_{h\in G}.$$ We usually denote this object by $(K^G,\sigma_G)$ and call it the full shift of the group $G$ over the \textit{alphabet} $K$.

Let $G$ act continuously on compact metrizable spaces $X_n$, respectively, where $n$ ranges over some subset $R$ of $\mathbb{N}$. We define the \textbf{product action} of $G$ on the product space $\prod_{n\in R}X_n$ as follows: $$g(x_n)_{n\in R}=(gx_n)_{n\in R},\quad\forall g\in G,\;\forall(x_n)_{n\in R}\in\prod_{n\in R}X_n.$$ Product actions will be considered in the sequel (e.g. the definition of sofic mean dimension and the proof of our main proposition).

\subsection{Sofic groups}
We denote by $|F|$ the cardinality of a set $F$. For every $d\in\mathbb{N}$ we write $[d]$ for the set $\{k\in\mathbb{N}:1\le k\le d\}$ and $\Sym(d)$ for the group of permutations of $[d]$. A group $G$ is \textbf{sofic} if there is a sequence $$\Sigma=\{\sigma_i:G\to\Sym(d_i)\}_{i\in\mathbb{N}}$$ (together with a sequence $\{d_i\}_{i\in\mathbb{N}}\subset\mathbb{N}$) such that the following three conditions are satisfied:
\begin{align*}
&\bullet\quad\quad\lim_{i\to\infty}\frac{1}{d_i}|\{k\in[d_i]:\sigma_i(st)(k)=\sigma_i(s)\sigma_i(t)(k)\}|=1\quad\text{for all $s,t\in G$;}\\
&\bullet\quad\quad\lim_{i\to\infty}\frac{1}{d_i}|\{k\in[d_i]:\sigma_i(s)(k)\ne\sigma_i(t)(k)\}|=1\quad\text{for all distinct $s,t\in G$;}\\
&\bullet\quad\quad\lim_{i\to\infty}d_i=+\infty.
\end{align*}
Such a sequence $\Sigma$ is called a \textbf{sofic approximation sequence} for $G$.
\begin{remark}
Note that the third condition will be fulfilled automatically if we additionally assume the group $G$ to be infinite.
\end{remark}

\subsection{$\epsilon$-embeddings and $\epsilon$-width dimension}
We denote by $\dim(K)$ the topological dimension (i.e. the Lebesgue covering dimension) of a compact metrizable space $K$. If the space $K$ is empty, then we set $\dim(K)=-\infty$.

Let $X$ and $P$ be compact metrizable spaces. Let $\rho$ be a compatible metric on $X$. For $\epsilon>0$ a continuous mapping $f:X\to P$ is called an \textbf{$\epsilon$-embedding with respect to $\rho$} if $f(x)=f(x^\prime)$ implies $\rho(x,x^\prime)<\epsilon$, for all $x,x^\prime\in X$. Let $\Widim_\epsilon(X,\rho)$ be the minimum (topological) dimension $\dim(P)$ of a compact metrizable space $P$ which admits an $\epsilon$-embedding $f:X\to P$ with respect to $\rho$. This term was generally known as the $\epsilon$-width dimension.
\begin{remark}
We may verify that the dimension of any compact metrizable space $X$ can be recovered by: $\;\dim(X)=\lim_{\epsilon\to0}\Widim_\epsilon(X,\rho)$.
\end{remark}

\subsection{Distances $\rho_2$ and $\rho_\infty$}
Let $K$ be a compact metrizable space. Let $\rho$ be a compatible metric on $K$. For every $n\in\mathbb{N}$ we define two (different) compatible metrics $\rho_2$ and $\rho_\infty$ on the product space $K^n$ as follows: $$\rho_2\left((x_i)_{i\in[n]},(y_i)_{i\in[n]}\right)=\sqrt{\frac1n\sum_{i\in[n]}(\rho(x_i,y_i))^2},$$$$\rho_\infty\left((x_i)_{i\in[n]},(y_i)_{i\in[n]}\right)=\max_{i\in[n]}\rho(x_i,y_i).$$ We do not include $n\in\mathbb{N}$ in the notations $\rho_2$ and $\rho_\infty$ because this does not cause any ambiguity.

\subsection{Mean dimension}
A group $G$ is said to be \textbf{amenable} if there exists a sequence $\{F_n\}_{n\in\mathbb{N}}$ of nonempty finite subsets of $G$ such that for any $g\in G$ $$\lim_{n\to\infty}\frac{|F_n\triangle gF_n|}{|F_n|}=0.$$ Such a sequence $\{F_n\}_{n\in\mathbb{N}}$ is called a \textbf{F{\o}lner sequence} of the group $G$. Obviously, all the finite groups are amenable.

Next we state the definition of mean dimension for amenable group actions. We remark again that this notion is not involved in the main results of this paper. We put it here for theoretical completeness.

Let an amenable group $G$ act continuously on a compact metrizable space $X$. Take a F{\o}lner sequence $\{F_n\}_{n\in\mathbb{N}}$ of $G$ and a compatible metric $\rho$ on $X$. For a nonempty finite subset $F$ of $G$ we set $$\rho_F(x,x^\prime)=\rho_\infty\left((gx)_{g\in F},(gx^\prime)_{g\in F}\right),\quad\forall\,x,x^\prime\in X.$$ It is clear that $\rho_F$ also becomes a compatible metric on $X$. The \textbf{mean dimension} of $(X,G)$ is defined by $$\lim_{\epsilon\to0}\lim_{n\to\infty}\frac{\Widim_\epsilon(X,\rho_{F_n})}{|F_n|}.$$ It is well known that both of the limits in the above definition always exist, and that this value is independent of the choices of a F{\o}lner sequence $\{F_n\}_{n\in\mathbb{N}}$ of $G$ and a compatible metric $\rho$ on $X$.

\subsection{Sofic mean dimension}
Suppose that a sofic group $G$ acts continuously on a compact metrizable space $X$. Let $$\Sigma=\{\sigma_i:G\to\Sym(d_i)\}_{i\in\mathbb{N}}$$ be a sofic approximation sequence for $G$. We now equip $X$ with a compatible metric $\rho$ temporarily. For a finite subset $F$ of $G$, $\delta>0$ and a map $\sigma:G\to\Sym(d)$ (where $d\in\mathbb{N}$) we define $$\Map(\rho,F,\delta,\sigma)=\{\phi:[d]\to X:\rho_2(\phi\circ\sigma(s),s\phi)\le\delta,\,\forall s\in F\}.$$ We consider the set $\Map(\rho,F,\delta,\sigma)$ as a compact subspace of the product space $X^d$. In our context we usually write $\phi=(\phi_l)_{l\in[d]}\in X^d$ for $\phi:[d]\to X$ (i.e. for $l\in[d]$ we write $\phi_l$ for $\phi(l)$). The \textbf{sofic mean dimension} of $(X,G)$ with respect to $\Sigma$ is defined by $$\mdim_\Sigma(X,G)=\sup_{\epsilon>0}\inf_{F\subset G\text{ finite, }\,\delta>0}\limsup_{i\to\infty}\frac{\Widim_\epsilon\left(\Map(\rho,F,\delta,\sigma_i),\rho_\infty\right)}{d_i}.$$ A standard fact (similar to mean dimension) is that the definition of $\mdim_\Sigma(X,G)$ does not depend on the choice of compatible metrics $\rho$ on $X$. Nevertheless, as mentioned implicitly in Subsection 1.1, it is not clear if there is an example of a sofic approximation sequence $\Sigma^\prime$ different from $\Sigma$, which leads to a value $\mdim_{\Sigma^\prime}(X,G)$ different from $\mdim_\Sigma(X,G)$.
\begin{lemma}[{\cite[Section 2]{Li}}]\label{sofictimes}
Let a sofic group $G$ act continuously on compact metrizable spaces $X_n$, respectively, where $n$ runs over some $R\subset\mathbb{N}$. Let $\Sigma$ be a sofic approximation sequence for $G$. For the product action of $G$ on the product space $\prod_{n\in R}X_n$ we have an inequality for sofic mean dimension: $\,\mdim_\Sigma(\prod_{n\in R}X_n,G)\le\sum_{n\in R}\mdim_\Sigma(X_n,G)$.
\end{lemma}

\medskip

\section{Main proposition}
\subsection{Statement of the main proposition}
In this subsection we state our main proposition (Theorem \ref{main0}). The main theorems of this paper follow from this. The proof of Theorem \ref{main0} is located in the next two sections.

Let $X$ be a compact metrizable space. We take a compatible metric $\rho$ on $X$ arbitrarily and temporarily. We consider an expression with terms of $\epsilon$-width dimension: $$\lim_{\epsilon\to0}\lim_{n\to\infty}\frac{\Widim_\epsilon(X^n,\rho_\infty)}{n}.$$ Also, as an alternative, it is equal to $$\sup_{\epsilon>0}\inf_{n\in\mathbb{N}}\frac{\Widim_\epsilon(X^n,\rho_\infty)}{n}.$$ Clearly, this is a topological invariant of $X$. More precisely, both of the limits in the former always exist; and moreover, the defined value is independent of the choice of compatible metrics $\rho$ on $X$. The inner limit (in the first-mentioned expression) exists because the term $\Widim_\epsilon(X^n,\rho_\infty)$ is subadditive\footnote{Here we have used the classically-known fact that the dimension of the product of finitely many compact metrizable spaces is at most the sum of the dimension of those ones.} in $n$ (and therefore, this limit can be replaced by $\inf_{n\in\mathbb{N}}$), while the outer limit (in the first-mentioned expression) exists because $\Widim_\epsilon(X^n,\rho_\infty)$ is monotone in $\epsilon$ (and hence, it can be replaced with $\sup_{\epsilon>0}$).

\begin{theorem}\label{main0}\textup{(Main proposition).}
\begin{itemize}
\item[\textup{(i)}]
Let $G$ be a finite group and $\Sigma$ a sofic approximation sequence for $G$. Let $G$ act continuously on a compact metrizable space $X$. The following equality is true:
$$\mdim_\Sigma(X,G)=\frac{1}{|G|}\cdot\lim_{\epsilon\to0}\lim_{n\to\infty}\frac{\Widim_\epsilon(X^n,\rho_\infty)}{n}.$$
\item[\textup{(ii)}]
Let $G$ be a sofic group and $\Sigma$ a sofic approximation sequence for $G$. Let $X$ be a compact metrizable space. The following equality is true:
$$\mdim_\Sigma(X^G,\sigma_G)=\lim_{\epsilon\to0}\lim_{n\to\infty}\frac{\Widim_\epsilon(X^n,\rho_\infty)}{n}.$$
\end{itemize}
\end{theorem}

\medskip

\subsection{Proof of the main theorems}
We are now ready to prove our main results assuming Theorem \ref{main0}.
\begin{corollary}[=Theorem \ref{main1}]
If an amenable group $G$ acts continuously on a compact metrizable space $X$, and if $\Sigma$ and $\Sigma^\prime$ are two sofic approximation sequences for $G$, then we have $\mdim_\Sigma(X,G)=\mdim_{\Sigma^\prime}(X,G)$.
\end{corollary}
\begin{proof}
If $G$ is an infinite amenable group, then by \cite[Section 3]{Li} we have $\mdim_\Sigma(X,G)=\mdim_{\Sigma^\prime}(X,G)$ (both $\mdim_\Sigma(X,G)$ and $\mdim_{\Sigma^\prime}(X,G)$ are equal to the mean dimension of $(X,G)$). If $G$ is a finite group, then Statement (i) of Theorem \ref{main0} shows in particular that $\mdim_\Sigma(X,G)=\mdim_{\Sigma^\prime}(X,G)$.
\end{proof}
\begin{corollary}[=Theorem \ref{main2}]
Let $K$ be a compact metrizable space. Let $G$ and $G^\prime$ be sofic groups. Let $\Sigma$ and $\Sigma^\prime$ be sofic approximation sequences for $G$ and $G^\prime$, respectively. The following equality is true:
$$\mdim_\Sigma(K^G,\sigma_G)=\mdim_{\Sigma^\prime}(K^{G^\prime},\sigma_{G^\prime}).$$
\end{corollary}
\begin{proof}
This follows from Statement (ii) of Theorem \ref{main0}.
\end{proof}

\medskip

\subsection{Proof of the quantitative main theorems}
As corollaries of our main proposition in the quantitative aspects, we prove Theorem \ref{main3} and Theorem \ref{main4} in this subsection. To this aim we shall employ a practical lemma due to Tsukamoto \cite[Lemma 3.1]{Tsukamoto}, which applies to finite-dimensional compact metrizable spaces. As follows we state this lemma.
\begin{lemma}[{\cite[Lemma 3.1]{Tsukamoto}}]\label{widimlowerbound}
Let $K$ be a finite-dimensional compact metrizable space. Let $\rho$ be a compatible metric on $K$. Then there is some $\delta>0$ such that for all $n\in\mathbb{N}$ and all $0<\epsilon<\delta$ we have $\Widim_\epsilon(K^n,\rho_\infty)\ge n\cdot(\dim(K)-1)$.
\end{lemma}
As mentioned previously, the proofs of the main proposition (Theorem \ref{main0}) and our main results (Theorem \ref{main1} and Theorem \ref{main2}) do \textit{not} rely on Tsukamoto's lemma. So Lemma \ref{widimlowerbound} is borrowed \textit{only} in this subsection (in the proof of Lemma \ref{widimdim}, and in consequence, Corollary \ref{corollary1} and Corollary \ref{corollary2}). Now we begin with the following lemma which may be regarded as a limit version of Lemma \ref{widimlowerbound}.
\begin{lemma}\label{widimdim}
If $K$ is a finite-dimensional compact metrizable space and if $\rho$ is a compatible metric on $K$, then $$\lim_{\epsilon\to0}\lim_{n\to\infty}\frac{\Widim_\epsilon(K^n,\rho_\infty)}{n}=\inf_{n\in\mathbb{N}}\frac{\dim(K^n)}{n}.$$
\end{lemma}
\begin{proof}
First of all, we notice that the term $\dim(K^n)$ is subadditive in $n\in\mathbb{N}$, which implies that $$\lim_{n\to\infty}\frac{\dim(K^n)}{n}=\inf_{n\in\mathbb{N}}\frac{\dim(K^n)}{n}.$$ Since $\Widim_\epsilon(K^n,\rho_\infty)\le\dim(K^n)$, we have $$\lim_{\epsilon\to0}\lim_{n\to\infty}\frac{\Widim_\epsilon(K^n,\rho_\infty)}{n}\le\inf_{n\in\mathbb{N}}\frac{\dim(K^n)}{n}.$$ To see the converse direction, we take an $m\in\mathbb{N}$ arbitrarily and fix it temporarily. Note that the space $K$ is finite-dimensional, and thus, the product space $K^m$ is finite-dimensional as well. We apply Lemma \ref{widimlowerbound} to the product space $K^m$:
\begin{align*}
\lim_{\epsilon\to0}\lim_{n\to\infty}\frac{\Widim_\epsilon(K^n,\rho_\infty)}{n}
&=\lim_{\epsilon\to0}\lim_{n\to\infty}\frac{\Widim_\epsilon(K^{mn},\rho_\infty)}{mn}\\
&\ge\lim_{\epsilon\to0}\lim_{n\to\infty}\frac{n\cdot(\dim(K^m)-1)}{mn}\\
&=\frac{\dim(K^m)-1}{m}.
\end{align*}
Since $m\in\mathbb{N}$ is arbitrary, the statement follows.
\end{proof}
\begin{corollary}[=Theorem \ref{main3}]\label{corollary1}
If a finite group $G$ acts continuously on a finite-dimensional compact metrizable space $X$, and if $\Sigma$ is a sofic approximation sequence for $G$, then $$\mdim_\Sigma(X,G)=\frac{1}{|G|}\cdot\inf_{n\in\mathbb{N}}\frac{\dim(X^n)}{n}.$$
\end{corollary}
\begin{proof}
This follows from Statement (i) of Theorem \ref{main0} and Lemma \ref{widimdim}.
\end{proof}
\begin{corollary}[=Theorem \ref{main4}]\label{corollary2}
If $K$ is a finite-dimensional compact metrizable space, and if $\Sigma$ is a sofic approximation sequence for a sofic group $G$, then $$\mdim_\Sigma(K^G,\sigma_G)=\inf_{n\in\mathbb{N}}\frac{\dim(K^n)}{n}.$$
\end{corollary}
\begin{proof}
This follows from Statement (ii) of Theorem \ref{main0} and Lemma \ref{widimdim}.
\end{proof}

\medskip

\subsection{Some remarks}
Let $K$ be any finite-dimensional compact metrizable space. We consider $(K^\mathbb{Z},\sigma_\mathbb{Z})$, the full shift of the group $\mathbb{Z}$ over the alphabet $K$. The following two types of upper bounds: $\dim(K)$ and $\inf_{n\in\mathbb{N}}(\dim(K^n)/n)$, for the mean dimension of $(K^\mathbb{Z},\sigma_\mathbb{Z})$, can be deduced directly from each other. More precisely, the inequality $\dim(K^n)/n\le\dim(K)$ (for $n\in\mathbb{N}$) is generally true. Conversely, if we assume that the mean dimension of $(K^\mathbb{Z},\sigma_\mathbb{Z})$ is not greater than $\dim(K)$ for any finite-dimensional alphabet $K$, then replacing $K$ by $K^n$ (which is finite-dimensional as well) for an arbitrary $n\in\mathbb{N}$ and by means of \cite{JQ} (or noting that $((K^n)^\mathbb{Z},\sigma_\mathbb{Z})$ is dynamically isomorphic to $(K^\mathbb{Z},\sigma_\mathbb{Z}^n)$) we see that the mean dimension of $(K^\mathbb{Z},\sigma_\mathbb{Z})$ is bounded from above by $\dim(K^n)/n$ for all $n\in\mathbb{N}$, and thus, is bounded by $\inf_{n\in\mathbb{N}}(\dim(K^n)/n$).

Here the point is that the acting group $\mathbb{Z}$ contains a subgroup $n\mathbb{Z}$ of index $n\in\mathbb{N}$ which is possible to be arbitrarily large, and meanwhile, any subgroup action of $\mathbb{Z}$ (i.e., a group action on the same space induced by some subgroup of $\mathbb{Z}$) also forms a $\mathbb{Z}$-action (which may be different from the original $\mathbb{Z}$-action). Unfortunately, the same trick does not apply to the context of sofic group actions essentially, because in general we cannot expect an (arbitrary) acting group to have such a nice property (even to admit a subgroup of arbitrarily large index) and also because we do not have such an equality (as established in the\footnote{We do not state this equality here as we shall not use it in the present paper.} paper \cite{JQ} by the authors) for sofic mean dimension of product actions with respect to the \textit{same} sofic approximation sequence. In consequence, Li's handling of the estimate $\mdim_\Sigma(K^G,\sigma_G)\le\dim(K)$ is, to a certain extent, not adequate for the purpose of obtaining its exact value. This turns out to be one of the main obstacles that we have to overcome. To deal with it in a more careful way is crucial to our proof.

\medskip

Further, from this point of view, we provide a (very) partial answer to the problems mentioned in Subsection 1.4. But we have to note that the assumption in the statement excludes some (extremely) wild cases. We have the following proposition:
\begin{itemize}\item
If a compact metrizable space $X$ admits a homeomorphism $T:X\to X$ such that the $\mathbb{Z}$-action on $X$ generated by $T$ has positive mean dimension, then all the values below (which are concerned in the equalities of the main proposition, where $\Sigma$ is a sofic approximation sequence for an acting group $G$) are equal to $+\infty$:
\begin{itemize}
\item$\mdim_\Sigma(X,G)$ for finite groups $G$ (acting continuously on $X$),
\item$\mdim_\Sigma(X^G,\sigma_G)$ for sofic groups $G$.
\end{itemize}
\end{itemize}
The proof is simple: The full shift $(X^\mathbb{Z},\sigma_\mathbb{Z})$ contains (as subshifts defined by orbits) all the possible $\mathbb{Z}$-actions on $X$, in particular, the $\mathbb{Z}$-action on $X$ generated by $T^n:X\to X$, for any $n\in\mathbb{N}$. This implies that the (sofic) mean dimension of $(X^\mathbb{Z},\sigma_\mathbb{Z})$ is equal to $+\infty$, and hence (by the main proposition), so are all those values.

\medskip

\subsection{Looking at the simplest case: What are the key differences?}
In this subsection we shall consider a toy-model of Theorem \ref{main0}. Notice that it is logically independent of all the other parts of this paper.

We denote by $\{e\}$ the trivial group (i.e. the group that consists of the identity element only). Let $\Sigma=\{\sigma_i:\{e\}\to\Sym(d_i)\}_{i\in\mathbb{N}}$ be a sofic approximation sequence for $\{e\}$. Let $X$ be a compact metrizable space. We consider the shift action of $\mathbb{Z}$ on the product space $X^\mathbb{Z}$ and the (uniquely possible) action of $\{e\}$ on the space $X$. We aim to show the following proposition:
\begin{itemize}\item
Both the mean dimension of $(X^\mathbb{Z},\sigma_\mathbb{Z})$ and the sofic mean dimension of $(X,\{e\})$ with respect to $\Sigma$ are equal to $$\lim_{\epsilon\to0}\lim_{n\to\infty}\frac{\Widim_\epsilon(X^n,\rho_\infty)}{n},$$
where $\rho$ is an arbitrarily fixed compatible metric on $X$.
\end{itemize}

\medskip

As follows we prove this proposition.

Firstly, we note that $$\lim_{i\to\infty}\frac{1}{d_i}|\{k\in[d_i]:\sigma_i(e)(k)=k\}|=1$$ which means that for $\delta>0$ we have $$\Map(\rho,\{e\},\delta,\sigma_i)=\{\phi:[d_i]\to X:\rho_2(\phi\circ\sigma_i(e),\phi)\le\delta\}=X^{d_i},$$ for all sufficiently large $i\in\mathbb{N}$, which implies that $$\mdim_\Sigma(X,\{e\})=\sup_{\epsilon>0}\limsup_{i\to\infty}\frac{\Widim_\epsilon(X^{d_i},\rho_\infty)}{d_i}=\lim_{\epsilon\to0}\lim_{n\to\infty}\frac{\Widim_\epsilon(X^n,\rho_\infty)}{n}.$$ Notice that here we used $d_i\to+\infty$ as $i\to\infty$. This proves the latter part of the above proposition.

Next we show the former part of this proposition. Since $X$ is compact, we may assume, without loss of generality, that the distance between any two points in $X$ is bounded by $1$ with respect to $\rho$. Let $D$ be the compatible metric on the product space $X^\mathbb{Z}$, defined as follows: $$D(x,x^\prime)=\sum_{i\in\mathbb{Z}}\frac{\rho(x_i,x^\prime_i)}{2^{|i|}}\quad\quad(x=(x_i)_{i\in\mathbb{Z}},x^\prime=(x^\prime_i)_{i\in\mathbb{Z}}\,\in X^\mathbb{Z}).$$ For every $n\in\mathbb{N}$ we write $D_{(n)}$ for the compatible metric on $X^\mathbb{Z}$ defined by $$D_{(n)}(x,x^\prime)=D_\infty(((x_{i+j})_{i\in\mathbb{Z}})_{j=0}^{n-1},((x^\prime_{i+j})_{i\in\mathbb{Z}})_{j=0}^{n-1})\quad\quad(x=(x_i)_{i\in\mathbb{Z}},x^\prime=(x^\prime_i)_{i\in\mathbb{Z}}\,\in X^\mathbb{Z}).$$ Now we fix a point $p\in X$. For every $n\in\mathbb{N}$ we define a mapping $f_n:X^n\to X^\mathbb{Z}$ by sending each $(x_i)_{i=0}^{n-1}\in X^n$ to $((p)_{i\le-1},(x_i)_{0\le i\le n-1},(p)_{i\ge n})\in X^\mathbb{Z}$. Clearly, the mapping $f_n:X^n\to X^\mathbb{Z}$ is continuous and distance-increasing with respect to the metrics $\rho_\infty$ on $X^n$ and $D_{(n)}$ on $X^\mathbb{Z}$, namely $$\rho_\infty((x_i)_{i=0}^{n-1},(x^\prime_i)_{i=0}^{n-1})\le D_{(n)}(f_n((x_i)_{i=0}^{n-1}),f_n((x^\prime_i)_{i=0}^{n-1})),\;\quad\forall\,(x_i)_{i=0}^{n-1},(x^\prime_i)_{i=0}^{n-1}\in X^n.$$ This implies that $$\Widim_\epsilon(X^n,\rho_\infty)\le\Widim_\epsilon(X^\mathbb{Z},D_{(n)}),\;\quad\forall n\in\mathbb{N},\,\,\forall\epsilon>0.$$ Thus, the mean dimension of $(X^\mathbb{Z},\sigma_\mathbb{Z})$ is bounded from below by $$\lim_{\epsilon\to0}\lim_{n\to\infty}\frac{\Widim_\epsilon(X^n,\rho_\infty)}{n}.$$ To estimate the mean dimension of $(X^\mathbb{Z},\sigma_\mathbb{Z})$ from above, we fix $\epsilon>0$ arbitrarily. Let $M\in\mathbb{N}$ be a constant (depending on $\epsilon$) such that $$\rho_\infty((x_i)_{i=-M}^M,(x^\prime_i)_{i=-M}^M)<\epsilon\Longrightarrow D(x,x^\prime)<2\epsilon,\;\quad\forall\,x=(x_i)_{i\in\mathbb{Z}},x^\prime=(x^\prime_i)_{i\in\mathbb{Z}}\in X^\mathbb{Z}.$$ For any $n\in\mathbb{N}$ we take a compact metrizable space $P_n$ satisfying $$\dim(P_n)=\Widim_\epsilon(X^{n+2M},\rho_\infty)$$ in company with a continuous mapping $f_n:X^{n+2M}\to P_n$ which is an $\epsilon$-embedding with respect to the metric $\rho_\infty$ on $X^{n+2M}$. Let $$\pi_n:X^\mathbb{Z}\to X^{n+2M},\quad\,(x_i)_{i\in\mathbb{Z}}\mapsto(x_i)_{i=-M}^{n+M-1}$$ be a canonical projection. Obviously, the mapping $\pi_n$ is continuous. It follows that the continuous mapping $f_n\circ\pi_n:X^\mathbb{Z}\to P_n$ is a $(2\epsilon)$-embedding with respect to the metric $D_{(n)}$ on $X^\mathbb{Z}$. Thus, $$\Widim_{2\epsilon}(X^\mathbb{Z},D_{(n)})\le\dim(P_n)=\Widim_\epsilon(X^{n+2M},\rho_\infty),\,\quad\forall n\in\mathbb{N}.$$ This implies that $$\lim_{n\to\infty}\frac{\Widim_{2\epsilon}(X^\mathbb{Z},D_{(n)})}{n}\le\lim_{n\to\infty}\frac{\Widim_\epsilon(X^{n+2M},\rho_\infty)}{n}=\lim_{n\to\infty}\frac{\Widim_\epsilon(X^n,\rho_\infty)}{n}.$$ Since $\epsilon>0$ is arbitrary, we obtain (by definition) that the mean dimension of $(X^\mathbb{Z},\sigma_\mathbb{Z})$ is bounded from above by $$\lim_{\epsilon\to0}\lim_{n\to\infty}\frac{\Widim_\epsilon(X^n,\rho_\infty)}{n}.$$ This ends the proof of the proposition.

\medskip

Although the proof of the above proposition is simple (and self-contained), it unifies the estimates from above and below, and hence explains how we relate mean dimension of full shifts to sofic mean dimension of finite group actions with a common expression. In particular, as we can see in the proof, all the compact metrizable spaces have now been treated with a unified process (i.e., the space $X$ in this proposition is not assumed to be finite-dimensional any more).

In spite of that, we have to be faced with extra obstacles to our estimates for sofic mean dimension if we generally replace the acting group $\mathbb{Z}$ in this proposition with a sofic (but non-amenable) group, or if we investigate the sofic mean dimension of a finite (but nontrivial) group action. More precisely, we will encounter difficulties arising mainly from the acting (sofic) groups due to some bad behaviour of sofic approximation sequences which differ substantially from F{\o}lner sequences which are used to characterize amenable groups that are endowed with a much nicer approximation structure. Section 4 and Section 5 are inspired by this cause.

\medskip

\section{Sofic mean dimension of finite group actions}
In this section we prove Statement (i) of Theorem \ref{main0}. To begin with, we put some general settings.

Let $X$ be a compact metrizable space. We take a compatible metric $\rho$ on $X$ arbitrarily. Since $X$ is compact, we may assume, for simplicity, that the diameter of $X$ with respect to $\rho$ is equal to $1$.

Let $G$ be a finite group. We fix a sofic approximation sequence $\Sigma=\{\sigma_i:G\to\Sym(d_i)\}_{i\in\mathbb{N}}$ for $G$.

Now we suppose that $G$ acts continuously on $X$. Our aim is to prove the following equality: $$\mdim_\Sigma(X,G)=\frac{1}{|G|}\cdot\lim_{\epsilon\to0}\lim_{n\to\infty}\frac{\Widim_\epsilon(X^n,\rho_\infty)}{n}.$$

This will be fulfilled by Lemma \ref{estimate-i-1} and Lemma \ref{estimate-i-2}.

\begin{lemma}[Estimate from above]\label{estimate-i-1}$$\mdim_\Sigma(X,G)\le\frac{1}{|G|}\cdot\lim_{\epsilon\to0}\lim_{n\to\infty}\frac{\Widim_\epsilon(X^n,\rho_\infty)}{n}.$$\end{lemma}
\begin{proof}
To show the statement, we fix $\epsilon>0$ and $\theta>0$ arbitrarily. Since $X$ is a compact metrizable space and since $G$ is a finite group, there is some $0<\eta=\eta(\epsilon)<\epsilon$ such that if two points $x,x^\prime\in X$ satisfy $\rho(x,x^\prime)\le3\eta$ then they must satisfy $\rho(gx,gx^\prime)<\epsilon$ for all $g\in G$. Clearly, we have $\eta=\eta(\epsilon)\to0$ as $\epsilon\to0$, and hence, by definition it suffices to prove (for those $\theta>0$, $\epsilon>0$ and $\eta>0$ fixed already) that there is some $\delta_0>0$ satisfying the following inequality: $$\limsup_{i\to\infty}\frac{\Widim_\epsilon(\Map(\rho,G,\delta_0,\sigma_i),\rho_\infty)}{d_i}\le\frac{1}{|G|}\cdot\lim_{n\to\infty}\frac{\Widim_\eta(X^n,\rho_\infty)}{n}+\theta.$$

For every $i\in\mathbb{N}$ we let $Q_i\subset[d_i]$ be the intersection of the following two sets: $$\left\{j\in[d_i]:\sigma_i(g)(j)\ne\sigma_i(h)(j),\;\forall g\ne h\in G\right\},$$$$\left\{j\in[d_i]:\sigma_i(g)(\sigma_i(h)(j))=\sigma_i(gh)(j),\;\forall g,h\in G\right\}.$$ For any $j\in[d_i]$ and $A\subset[d_i]$ we set $$\sigma_i(G)(j)=\{\sigma_i(g)(j):g\in G\}\subset[d_i],\quad\sigma_i(G)(A)=\{\sigma_i(g)(a):g\in G,a\in A\}\subset[d_i].$$ We note that by the definition of $Q_i$ we have $|\sigma_i(G)(j)|=|G|$ and $\sigma_i(G)(\sigma_i(G)(j))=\sigma_i(G)(j)$, for all $j\in Q_i$. It follows that $$\sigma_i(G)(j)=\sigma_i(G)(l),\quad\quad\forall\,l\in\sigma_i(G)(j).$$ We put an element $j\in Q_i$ into a set $C_i$ and remove all the members of $\sigma_i(G)(j)$ from the set $Q_i$. By dealing with the resulting sets (with the same method that we just described) finitely many times we can find\footnote{Here the argument is simple because the acting group $G$ is finite. It is also possible to have such an injective mapping, carried out with a much more complicated process, if the group is amenable but not necessarily finite. We do not need this in our proof. For the general statement, please refer to \cite[Section 3]{Li}.} some $C_i\subset Q_i$ of the maximum cardinality satisfying that the mapping $$G\times C_i\to[d_i],\quad(g,j)\mapsto\sigma_i(g)(j)$$ is injective. We note that this implies automatically that $\sigma_i(e)(c)=c$, for all $c\in C_i$, where $e$ is the identity element of the group $G$.

For every $n\in\mathbb{N}$ we take a compact metrizable space $Y_n$ satisfying $$\Widim_\eta(X^n,\rho_\infty)=\dim(Y_n)<+\infty$$ in company with a continuous mapping $$f_n:X^n\to Y_n$$ which is an $\eta$-embedding with respect to the metric $\rho_\infty$ on the product space $X^n$. We notice that the space $Y_n$ is finite-dimensional because the space $X^n$ is compact.

We take $\tau>0$ with $$\tau\cdot\Widim_\eta(X,\rho)<\theta/2.$$ Since $G$ is finite, there exists some $i_0\in\mathbb{N}$ such that $$|\sigma_i(G)(C_i)|>(1-\tau)\cdot d_i,\quad\forall i\ge i_0.$$ We take $\delta_0>0$ satisfying $$\delta_0<\min\left\{\eta^2,\,\frac{1}{|G|\cdot(1+\Widim_\eta(X,\rho))}\cdot\frac{\theta}{2}\right\}.$$

Let $CY_1=([0,1]\times Y_1)/\sim$ be the cone generated by $Y_1$, where $(0,y)\sim(0,y^\prime)$ for any $y,y^\prime\in Y_1$. We denote by $\lambda y$ the equivalence class of $(\lambda,y)\in[0,1]\times Y_1$. In particular, we set $\ast=0y$ (for all $y\in Y_1$). The symbol $\ast$ is to denote the vertex of the cone $CY_1$. The following fact is obvious: $$\Widim_\eta(X,\rho)=\dim(Y_1)\le\dim(CY_1)\le1+\dim(Y_1)=1+\Widim_\eta(X,\rho).$$

We fix $i\in\mathbb{N}$ for the moment.

For each $j\in[d_i]$ we define a \textit{continuous} mapping $$J_j:\Map(\rho,G,\delta_0,\sigma_i)\to[0,1]$$ by sending $\phi=(\phi_j)_{j\in[d_i]}\in\Map(\rho,G,\delta_0,\sigma_i)$ to $$J_j(\phi)=\max\left\{\max_{g\in G}\left(\rho(g\phi_j,\phi_{\sigma_i(g)(j)})-\sqrt{\delta_0}\right),\,0\right\}.$$ Now we define a mapping $H_i$ as follows: $$H_i:\Map(\rho,G,\delta_0,\sigma_i)\to Y_{|C_i|}\times(CY_1)^{d_i}\times Y_1^{|[d_i]\setminus\sigma_i(G)(C_i)|},$$$$\phi=(\phi_j)_{j\in[d_i]}\mapsto\left(f_{|C_i|}((\phi_j)_{j\in C_i}),(J_j(\phi)f_1(\phi_j))_{j\in[d_i]},(f_1(\phi_j))_{j\in[d_i]\setminus\sigma_i(G)(C_i)}\right).$$ Since all the mappings $J_j$ (where $j\in[d_i]$) and $f_n$ (where $n\in\mathbb{N}$) are continuous, the mapping $H_i$ constructed above is also continuous.

We claim that $H_i$ is an $\epsilon$-embedding with respect to $\rho_\infty$. To verify this, we suppose that $H_i(\xi)=H_i(\psi)$ for some $\xi,\psi\in\Map(\rho,G,\delta_0,\sigma_i)$. What we need to show is that $\rho(\xi_j,\psi_j)<\epsilon$ for all $j\in[d_i]$. If $j\in[d_i]$ satisfies $J_j(\xi)=J_j(\psi)>0$ or it satisfies $j\in[d_i]\setminus\sigma_i(G)(C_i)$, then it is clear that $f_1(\xi_j)=f_1(\psi_j)$ which (by noting the fact that the continuous mapping $f_1:X\to Y_1$ is an $\eta$-embedding with respect to the metric $\rho$ on $X$) implies that $\rho(\xi_j,\psi_j)<\eta<\epsilon$. So we now assume that $j\in\sigma_i(G)(C_i)$ and $J_j(\xi)=J_j(\psi)=0$. This implies that $j=\sigma_i(s)(c)$ for some $s\in G$ and $c\in C_i\subset Q_i$, and that $$\rho(s^{-1}\xi_j,\xi_{\sigma_i(s^{-1})(j)})\le\sqrt{\delta_0}<\eta,\quad\quad\rho(s^{-1}\psi_j,\psi_{\sigma_i(s^{-1})(j)})\le\sqrt{\delta_0}<\eta.$$ Since $f_{|C_i|}((\xi_k)_{k\in C_i})=f_{|C_i|}((\psi_k)_{k\in C_i})$ and since the continuous mapping $f_{|C_i|}:X^{C_i}\to Y_{|C_i|}$ is an $\eta$-embedding with respect to the metric $\rho_\infty$ on the product space $X^{C_i}$, we have $\rho_\infty((\xi_k)_{k\in C_i},(\psi_k)_{k\in C_i})<\eta$, and hence in particular, $\rho(\xi_c,\psi_c)<\eta$. By noting that $$\sigma_i(s^{-1})(j)=\sigma_i(s^{-1})(\sigma_i(s)(c))=\sigma_i(s^{-1}s)(c)=\sigma_i(e)(c)=c,$$ we have $$\rho(\xi_{\sigma_i(s^{-1})(j)},\psi_{\sigma_i(s^{-1})(j)})<\eta.$$ It follows that $$\rho(s^{-1}\xi_j,s^{-1}\psi_j)\le3\eta.$$ Thus, we obtain $\rho(\xi_j,\psi_j)<\epsilon$. This proves the claim.

It follows from this claim that $\Widim_\epsilon(\Map(\rho,G,\delta_0,\sigma_i),\rho_\infty)$ is bounded from above by $\dim(H_i(\Map(\rho,G,\delta_0,\sigma_i)))$. To estimate the latter term, we consider the set $$P_i(\phi)=\{j\in[d_i]:J_j(\phi)>0\},$$ for any $\phi\in\Map(\rho,G,\delta_0,\sigma_i)$. When $\phi\in\Map(\rho,G,\delta_0,\sigma_i)$ we have $$\delta_0\cdot|P_i(\phi)|\le\sum_{g\in G,j\in[d_i]}\rho(g\phi_j,\phi_{\sigma_i(g)(j)})^2\le|G|\cdot\delta_0^2\cdot d_i$$ and hence $$|P_i(\phi)|\le|G|\cdot\delta_0\cdot d_i.$$ This implies that any element in the image (regarded as a compact subset of the product space $(CY_1)^{d_i}$ consisting of members all having $d_i$ entries) of the mapping $$\Map(\rho,G,\delta_0,\sigma_i)\to(CY_1)^{d_i},\quad\quad(\phi_j)_{j\in[d_i]}\mapsto(J_j(\phi)f_1(\phi_j))_{j\in[d_i]}$$ has entries at most $|G|\cdot\delta_0\cdot d_i$, which do not take the value $\ast$. Since the dimension of a finite union of compact metrizable spaces is equal to the maximum among the dimension of those participants in the union, we conclude that $\dim(H_i(\Map(\rho,G,\delta_0,\sigma_i)))$ is bounded from above by $$\dim(Y_{|C_i|})+\dim(CY_1)\cdot|G|\cdot\delta_0\cdot d_i+\dim(Y_1)\cdot|[d_i]\setminus\sigma_i(G)(C_i)|$$ which, for all $i\ge i_0$, does not exceed
\begin{align*}
&\dim(Y_{|C_i|})+(1+\dim(Y_1))\cdot|G|\cdot\delta_0\cdot d_i+\dim(Y_1)\cdot\tau\cdot d_i\\
=&\Widim_\eta(X^{|C_i|},\rho_\infty)+(1+\Widim_\eta(X,\rho))\cdot|G|\cdot\delta_0\cdot d_i+\Widim_\eta(X,\rho)\cdot\tau\cdot d_i\\
\le&\Widim_\eta(X^{|C_i|},\rho_\infty)+\frac\theta2\cdot d_i+\frac\theta2\cdot d_i\\
=&\Widim_\eta(X^{|C_i|},\rho_\infty)+\theta\cdot d_i.
\end{align*}
Finally, we note that $|C_i|\to+\infty$ (because $d_i\to+\infty$) as $i\to\infty$ and that for all $i\in\mathbb{N}$ we have $d_i\ge|G|\cdot|C_i|$. Thus, we deduce that
\begin{align*}
\limsup_{i\to\infty}\frac{\Widim_\epsilon(\Map(\rho,G,\delta_0,\sigma_i),\rho_\infty)}{d_i}
&\le\limsup_{i\to\infty}\frac{\dim(H_i(\Map(\rho,G,\delta_0,\sigma_i)))}{d_i}\\
&\le\limsup_{i\to\infty}\frac{\Widim_\eta(X^{|C_i|},\rho_\infty)}{d_i}+\theta\\
&\le\frac{1}{|G|}\cdot\limsup_{i\to\infty}\frac{\Widim_\eta(X^{|C_i|},\rho_\infty)}{|C_i|}+\theta\\
&=\frac{1}{|G|}\cdot\lim_{n\to\infty}\frac{\Widim_\eta(X^n,\rho_\infty)}{n}+\theta.
\end{align*}
This completes the proof.
\end{proof}

\begin{lemma}[Estimate from below]\label{estimate-i-2}$$\mdim_\Sigma(X,G)\ge\frac{1}{|G|}\cdot\lim_{\epsilon\to0}\lim_{n\to\infty}\frac{\Widim_\epsilon(X^n,\rho_\infty)}{n}.$$\end{lemma}
\begin{proof}
First of all, we note that Lemma \ref{sofictimes} allows us to reduce this issue to the following statement: $$\mdim_\Sigma(X^{|G|},G)\ge\lim_{\epsilon\to0}\lim_{n\to\infty}\frac{\Widim_\epsilon(X^n,\rho_\infty)}{n}.$$ To show the reduced inequality, we fix the compatible metric $D=\rho_\infty$ on the product space $X^{|G|}$. We would like to remind the reader to keep in mind that only the product action $(X^{|G|},G)$ is concerned in the remaining part of this proof.

We take $\epsilon>0$ and $\delta>0$ arbitrarily and fix them in the proof. By definition it suffices to prove that $$\limsup_{i\to\infty}\frac{\Widim_\epsilon(\Map(D,G,\delta,\sigma_i),D_\infty)}{d_i}\ge\lim_{n\to\infty}\frac{\Widim_\epsilon(X^n,\rho_\infty)}{n}.$$

For every $i\in\mathbb{N}$ we let $Q_i\subset[d_i]$ be the intersection of the following two sets: $$\left\{j\in[d_i]:\sigma_i(g)(j)\ne\sigma_i(g^\prime)(j),\;\forall g\ne g^\prime\in G\right\},$$$$\left\{j\in[d_i]:\sigma_i(g)(\sigma_i(g^\prime)(j))=\sigma_i(gg^\prime)(j),\;\forall g,g^\prime\in G\right\}.$$ For any $j\in[d_i]$ and $A\subset[d_i]$ we set $$\sigma_i(G)(j)=\{\sigma_i(g)(j):g\in G\}\subset[d_i],\quad\sigma_i(G)(A)=\{\sigma_i(g)(a):g\in G,a\in A\}\subset[d_i].$$ We note that by the definition of $Q_i$ we have $$|\sigma_i(G)(j)|=|G|,\quad\quad\sigma_i(G)(\sigma_i(G)(j))=\sigma_i(G)(j),\quad\quad\forall j\in Q_i.$$ It follows that $\sigma_i(G)(j)=\sigma_i(G)(l)$, for all $l\in\sigma_i(G)(j)$. As explained before (in the proof of Lemma \ref{estimate-i-1}), we can find a subset $C_i$ of $Q_i$, such that the mapping $$G\times C_i\to[d_i],\quad(g,j)\mapsto\sigma_i(g)(j)$$ is injective. Since $G$ is a finite group, there is some $i_0\in\mathbb{N}$ (sufficiently large) satisfying that $$|\sigma_i(G)(C_i)|>(1-\delta^2)\cdot d_i,\quad\quad\forall\;i\ge i_0.$$

For every $i\in\mathbb{N}$ we define a mapping $P_i$ as follows: $$P_i:X^{d_i}\to(X^{|G|})^{d_i},\quad p=(p_j)_{j\in[d_i]}\mapsto\phi=(\phi_j)_{j\in[d_i]}=((\phi_j(g))_{g\in G})_{j\in[d_i]},$$ where $\phi_j(g)$ is defined by $$\phi_j(g)=\begin{cases}hg^{-1}p_{\sigma_i(g)(k)},&\quad\text{ if }\;j=\sigma_i(h)(k),\;h\in G,\,k\in C_i\\p_j,&\quad\text{ if }\;j\in[d_i]\setminus\sigma_i(G)(C_i)\end{cases}.$$ Notice that here we have identified $X^{|G|}$ with $X^G$ for the indices' convenience.

The mapping $P_i:X^{d_i}\to(X^{|G|})^{d_i}$ is well-defined, because the mapping $$G\times C_i\to[d_i],\quad(g,j)\mapsto\sigma_i(g)(j)$$ is injective. Clearly, it is continuous. Since for any $j\in[d_i]$ there is some $g\in G$ such that $\phi_j(g)=p_j$, the mapping $P_i:X^{d_i}\to(X^{|G|})^{d_i}$ is distance-increasing with respect to the metric $\rho_\infty$ on $X^{d_i}$ and the metric $D_\infty$ on $(X^{|G|})^{d_i}$, i.e. $$\rho_\infty(p,p^\prime)\le D_\infty(P_i(p),P_i(p^\prime)),\quad\quad\forall p,p^\prime\in X^{d_i}.$$

We now claim that for all sufficiently large $i\in\mathbb{N}$ it will be true that $P_i(X^{d_i})$ is contained in $\Map(D,G,\delta,\sigma_i)$. This claim will be verified in a moment. Notice the fact that $d_i\to+\infty$ as $i\to\infty$. By the claim we will obtain that
\begin{align*}
\limsup_{i\to\infty}\frac{\Widim_\epsilon(\Map(D,G,\delta,\sigma_i),D_\infty)}{d_i}
&\ge\limsup_{i\to\infty}\frac{\Widim_\epsilon(P_i(X^{d_i}),D_\infty)}{d_i}\\
&\ge\limsup_{i\to\infty}\frac{\Widim_\epsilon(X^{d_i},\rho_\infty)}{d_i}\\
&=\lim_{n\to\infty}\frac{\Widim_\epsilon(X^n,\rho_\infty)}{n}.
\end{align*}
This will end the proof.

In what follows we prove that $P_i(X^{d_i})$ is contained in $\Map(D,G,\delta,\sigma_i)$, for all integers $i\ge i_0$. We fix an integer $i\ge i_0$. We take an arbitrary $p=(p_j)_{j\in[d_i]}\in X^{d_i}$. Let $\phi=P_i(p)$ and write $$\phi=(\phi_j)_{j\in[d_i]}=((\phi_j(g))_{g\in G})_{j\in[d_i]}\;\in\,(X^{|G|})^{d_i}=(X^{|G|})^{d_i}.$$ We note that $C_i\subset Q_i$. By the construction of the mapping $P_i:X^{d_i}\to(X^{|G|})^{d_i}$ we have obviously that if $j\in[d_i]$ satisfies $j=\sigma_i(h)(c)$ for some $h\in G$ and some $c\in C_i$ then for any $s\in G$ and any $g\in G$ $$s\phi_j(g)=shg^{-1}p_{\sigma_i(g)(c)}=\phi_{\sigma_i(sh)(c)}(g)=\phi_{\sigma_i(s)(\sigma_i(h)(c))}(g)=\phi_{\sigma_i(s)(j)}(g).$$ Thus, for any $j\in\sigma_i(G)(C_i)$ $$s\phi_j=\phi_{\sigma_i(s)(j)},\quad\quad\forall\,s\in G.$$ It follows that for all $s\in G$
\begin{align*}
D_2(s\phi,\phi\circ\sigma_i(s))&=\sqrt{\frac{1}{d_i}\cdot\sum_{j\in[d_i]}\left(D(\phi_{\sigma_i(s)(j)},s\phi_j)\right)^2}\\
&=\sqrt{\frac{1}{d_i}\cdot\sum_{j\in[d_i]\setminus\sigma_i(G)(C_i)}\left(D(\phi_{\sigma_i(s)(j)},s\phi_j)\right)^2}\\
&\le\sqrt{\frac{d_i-|\sigma_i(G)(C_i)|}{d_i}}\\
&\le\delta.
\end{align*}
This implies that $\phi\in\Map(D,G,\delta,\sigma_i)$. Thus, we conclude.
\end{proof}

\medskip

\section{Sofic mean dimension of full shifts}
In this section we prove Statement (ii) of Theorem \ref{main0}. Let us start with some general settings.

Let $K$ be a compact metrizable space. We take a compatible metric $D$ on the alphabet $K$ arbitrarily. Let $G$ be a sofic group. Let $\Sigma=\{\sigma_i:G\to\Sym(d_i)\}_{i\in\mathbb{N}}$ be a sofic approximation sequence for $G$. We are dealing with the shift action $\sigma_G$ of $G$ on the product space $K^G$.

We denote by $e$ the identity element of the group $G$. We fix a (countable) family $\{\alpha_g\}_{g\in G}$ of positive real numbers such that $$\alpha_e=1,\quad\quad\sum_{g\in G}\alpha_g<2.$$ We fix a metric $\rho$ on $K^G$ compatible with the product topology as follows: $$\rho(x,y)=\sum_{g\in G}\alpha_gD(x_g,y_g),\quad(x=(x_g)_{g\in G},y=(y_g)_{g\in G}\in K^G).$$ Since $K$ is a compact metrizable space, so is $K^G$. Thus, we may assume, without loss of generality, that the diameter of the product space $K^G$ with respect to the metric $\rho$ is equal to $1$.

Our goal is to prove the following equality: $$\mdim_\Sigma(K^G,\sigma_G)=\lim_{\epsilon\to0}\lim_{n\to\infty}\frac{\Widim_\epsilon(K^n,D_\infty)}{n}.$$ This will follow directly from Lemma \ref{estimate-ii-1} and Lemma \ref{estimate-ii-2}.

\begin{lemma}[Estimate from below]\label{estimate-ii-1}$$\mdim_\Sigma(K^G,\sigma_G)\ge\lim_{\epsilon\to0}\lim_{n\to\infty}\frac{\Widim_\epsilon(K^n,D_\infty)}{n}.$$\end{lemma}
\begin{proof}
We take $\epsilon>0$ arbitrarily. We take $\delta>0$ and a (nonempty) finite subset $F$ of $G$. We fix them in the proof. By definition it suffices to show that $$\limsup_{i\to\infty}\frac{\Widim_\epsilon(\Map(\rho,F,\delta,\sigma_i),\rho_\infty)}{d_i}\ge\lim_{n\to\infty}\frac{\Widim_\epsilon(K^n,D_\infty)}{n}.$$

For every $i\in\mathbb{N}$ we let $$P_i:K^{d_i}\to(K^G)^{d_i},\quad\quad p=(p_j)_{j\in[d_i]}\mapsto\phi=(\phi_j)_{j\in[d_i]}$$ (where $\phi_j$ has the form $\phi_j=(\phi_j(g))_{g\in G}\in K^G$, for each $j\in[d_i]$) be the mapping defined by $$\phi_j(g)=\begin{cases}p_{\sigma_i(g)(j)},&\;\text{ if }\,g\in G\setminus\{e\}\\p_j,&\;\text{ if }\,g=e\end{cases}.$$

Clearly, for any $i\in\mathbb{N}$ the mapping $P_i:K^{d_i}\to(K^G)^{d_i}$ is continuous. Moreover, it is distance-increasing with respect to the metric $D_\infty$ on $K^{d_i}$ and the metric $\rho_\infty$ on $(K^G)^{d_i}$, i.e., $$D_\infty(p,p^\prime)\le\rho_\infty(P_i(p),P_i(p^\prime)),\quad\forall p,p^\prime\in K^{d_i}.$$ We observe that for all sufficiently large $i\in\mathbb{N}$ it is true that $P_i(K^{d_i})$ is a subset of $\Map(\rho,F,\delta,\sigma_i)$. We will verify this observation in a moment. Assuming the observation and noting that $d_i\to+\infty$ as $i\to\infty$ we will deduce that
\begin{align*}
\limsup_{i\to\infty}\frac{\Widim_\epsilon(\Map(\rho,F,\delta,\sigma_i),\rho_\infty)}{d_i}
&\ge\limsup_{i\to\infty}\frac{\Widim_\epsilon(P_i(K^{d_i}),\rho_\infty)}{d_i}\\
&\ge\limsup_{i\to\infty}\frac{\Widim_\epsilon(K^{d_i},D_\infty)}{d_i}\\
&=\lim_{n\to\infty}\frac{\Widim_\epsilon(K^n,D_\infty)}{n}.
\end{align*}
This will finally end the proof.

As follows we verify that for all sufficiently large $i\in\mathbb{N}$ we have that $P_i(K^{d_i})$ is contained in $\Map(\rho,F,\delta,\sigma_i)$. We choose a (nonempty) finite subset $E$ of $G$, containing the identity element $e\in G$, such that if two points $\xi=(\xi_g)_{g\in G}$ and $\xi^\prime=(\xi^\prime_g)_{g\in G}$ in $K^G$ satisfy $\xi_g=\xi^\prime_g$ for all $g\in E$, then they satisfy $\rho(\xi,\xi^\prime)<\delta/2$. For every $i\in\mathbb{N}$ we put $$Q_i=\left\{j\in[d_i]:\sigma_i(s)\circ\sigma_i(t)(j)=\sigma_i(st)(j),\,\;\forall s\in E,\,\forall t\in F\cup\{e\}\right\}.$$ Since both $E$ and $F$ are finite subsets of $G$, there is some $k\in\mathbb{N}$ sufficiently large such that for any integer $i>k$ $$\frac{|Q_i|}{d_i}>1-\frac{3\delta^2}{8}.$$ Now we prove $$P_i(K^{d_i})\subset\Map(\rho,F,\delta,\sigma_i),\quad\,\;\forall i>k.$$ We fix an integer $i>k$. We take $p\in K^{d_i}$ arbitrarily. Let $\phi=P_i(p)$. We write $$p=(p_j)_{j\in[d_i]}\in K^{d_i},\quad\;\quad\,\phi=((\phi_j(g))_{g\in G})_{j\in[d_i]}\in(K^G)^{d_i}.$$ It is clear that for any $j\in Q_i$ we have $\sigma_i(e)(j)=j$. Therefore, $$\phi_j(g)=p_{\sigma_i(g)(j)},\quad\quad\forall g\in G,\quad\forall j\in Q_i.$$ It follows that for any $j\in Q_i\cap(\sigma_i(t))^{-1}(Q_i)$ and any $t\in F\cup\{e\}$ $$t\phi_j(s)=\phi_j(st)=p_{\sigma_i(st)(j)}=p_{\sigma_i(s)\circ\sigma_i(t)(j)}=\phi_{\sigma_i(t)(j)}(s),\,\quad\forall s\in E.$$ By the choice of $E$ we have $$\rho(t\phi_j,\phi_{\sigma_i(t)(j)})<\delta/2,\quad\quad\forall t\in F,\quad\forall j\in Q_i\cap(\sigma_i(t))^{-1}(Q_i).$$ Thus, we conclude that for all $t\in F$
\begin{align*}
\rho_2(\phi\circ\sigma_i(t),t\phi)&=\sqrt{\frac{1}{d_i}\cdot\sum_{j\in[d_i]}\left(\rho(\phi_{\sigma_i(t)(j)},t\phi_j)\right)^2}\\
&\le\sqrt{\frac{\delta^2}{4}+\frac{\left|[d_i]\setminus(Q_i\cap(\sigma_i(t))^{-1}(Q_i))\right|}{d_i}}\\
&\le\sqrt{\frac{\delta^2}{4}+2\cdot\left(1-\frac{|Q_i|}{d_i}\right)}\\
&\le\delta.
\end{align*}
This implies that $\phi\in\Map(\rho,F,\delta,\sigma_i)$. The statement follows.
\end{proof}

\begin{lemma}[Estimate from above]\label{estimate-ii-2}$$\mdim_\Sigma(K^G,\sigma_G)\le\lim_{\epsilon\to0}\lim_{n\to\infty}\frac{\Widim_\epsilon(K^n,D_\infty)}{n}.$$\end{lemma}
\begin{proof}
We take $\epsilon>0$ and $\eta>0$ arbitrarily and fix them in the proof. By definition it suffices to prove that there exist a finite (and nonempty) subset $F_0$ of $G$ and some $\delta_0>0$ such that $$\limsup_{i\to\infty}\frac{\Widim_{3\epsilon}\left(\Map(\rho,F_0,\delta_0,\sigma_i),\rho_\infty\right)}{d_i}\le\eta+\lim_{n\to\infty}\frac{\Widim_\epsilon(K^n,D_\infty)}{n}.$$

To this aim, for any $n\in\mathbb{N}$ we take a compact metrizable space $A_n$ satisfying $$\dim(A_n)=\Widim_\epsilon(K^n,D_\infty)$$ in company with a continuous mapping $f_n:K^n\to A_n$ which is an $\epsilon$-embedding with respect to the metric $D_\infty$ on the product space $K^n$. Since $K$ is compact (and nonempty), the term $\Widim_\epsilon(K^n,D_\infty)$ is always finite.

We choose a finite subset $F_0$ of $G$, containing the identity element $e$, such that if two points $\xi=(\xi(g))_{g\in G}$ and $\xi^\prime=(\xi^\prime(g))_{g\in G}$ coming from $K^G$ satisfy $D(\xi(s),\xi^\prime(s))\le2\epsilon$ for all $s\in F_0$, then they satisfy $\rho(\xi,\xi^\prime)<3\epsilon$.

We pick $\delta_0>0$ sufficiently small such that $$\delta_0<\min\left\{\frac{\eta}{(1+\Widim_\epsilon(K,D))\cdot|F_0|^2},\;\frac{\epsilon^2}{4}\right\}.$$

So our main task now is to estimate $\Widim_{3\epsilon}(\Map(\rho,F_0,\delta_0,\sigma_i),\rho_\infty)$ from above, for all sufficiently large $i\in\mathbb{N}$.

We fix an arbitrary $i\in\mathbb{N}$ for the moment.

For every $j\in[d_i]$ we define a mapping $$J_j:\Map(\rho,F_0,\delta_0,\sigma_i)\to[0,1]$$ by sending $$\phi=(\phi_j)_{j\in[d_i]}\in\Map(\rho,F_0,\delta_0,\sigma_i)\subset(K^G)^{d_i}$$ to $$J_j(\phi)=\max\left\{\max_{s\in F_0}\left(\rho(s\phi_j,\phi_{\sigma_i(s)(j)})-\sqrt{\delta_0}\right),0\right\},$$ where each $\phi_j$ ($j\in[d_i]$) has the form $\phi_j=(\phi_j(g))_{g\in G}\in K^G$. Clearly, all the mappings $J_j:\Map(\rho,F_0,\delta_0,\sigma_i)\to[0,1]$ ($j\in[d_i]$) are continuous.

Let $CA_1=([0,1]\times A_1)/\sim$ be the cone generated by $A_1$, where $(0,a)\sim(0,a^\prime)$ for all $a,a^\prime\in A_1$. We denote by $\lambda a$ the equivalence class of $(\lambda,a)\in[0,1]\times A_1$. We denote the vertex of the cone $CA_1$ by the symbol $\ast$ (namely, we set $\ast=0a$, for any $a\in A_1$). We notice that the following inequality is clear: $$\Widim_\epsilon(K,D)=\dim(A_1)\le\dim(CA_1)\le1+\dim(A_1)=1+\Widim_\epsilon(K,D).$$

We construct a mapping $H_i$ as follows: $$H_i:\Map(\rho,F_0,\delta_0,\sigma_i)\to A_{d_i}\times(CA_1)^{|F_0|\cdot d_i},$$$$\phi=(\phi_j)_{j\in[d_i]}\,\mapsto\,\left(f_{d_i}((\phi_j(e))_{j\in[d_i]}),\;((J_j(\phi)f_1(\phi_j(s)))_{s\in F_0})_{j\in[d_i]}\right).$$ Obviously, the mapping $H_i:\Map(\rho,F_0,\delta_0,\sigma_i)\to A_{d_i}\times(CA_1)^{|F_0|\cdot d_i}$ is continuous (because each mapping $J_j:\Map(\rho,F_0,\delta_0,\sigma_i)\to[0,1]$, where $j\in[d_i]$, is continuous).

We claim that the mapping $H_i:\Map(\rho,F_0,\delta_0,\sigma_i)\to A_{d_i}\times(CA_1)^{|F_0|\cdot d_i}$ is a $(3\epsilon)$-embedding with respect to $\rho_\infty$. To verify this claim, we take two points $\varphi=(\varphi_j)_{j\in[d_i]}$ and $\psi=(\psi_j)_{j\in[d_i]}$ arbitrarily within $\Map(\rho,F_0,\delta_0,\sigma_i)$. We assume that $H_i(\varphi)=H_i(\psi)$. What we then need to show is that $\rho_\infty(\varphi,\psi)<3\epsilon$. In fact, it follows from the assumption $H_i(\varphi)=H_i(\psi)$ that $$f_{d_i}((\varphi_j(e))_{j\in[d_i]})=f_{d_i}((\psi_j(e))_{j\in[d_i]});$$$$(J_j(\varphi)f_1(\varphi_j(s)))_{s\in F_0}=(J_j(\psi)f_1(\psi_j(s)))_{s\in F_0},\quad\quad\forall j\in[d_i].$$ Since the continuous mapping $f_{d_i}:K^{d_i}\to A_{d_i}$ is an $\epsilon$-embedding with respect to the distance $D_\infty$, the former implies that $D_\infty((\varphi_j(e))_{j\in[d_i]},(\psi_j(e))_{j\in[d_i]})<\epsilon$, i.e. $$D(\varphi_j(e),\psi_j(e))<\epsilon,\quad\quad\forall j\in[d_i].$$ The latter sorts all those $j\in[d_i]$ into two cases:
\begin{itemize}
\item either $(J_j(\varphi)f_1(\varphi_j(s)))_{s\in F_0}=(J_j(\psi)f_1(\psi_j(s)))_{s\in F_0}=(\ast,\ast,\dots,\ast)$;
\item or $(J_j(\varphi)f_1(\varphi_j(s)))_{s\in F_0}=(J_j(\psi)f_1(\psi_j(s)))_{s\in F_0}\in(CA_1\setminus\{\ast\})^{|F_0|}$.
\end{itemize}
If some $j\in[d_i]$ encounters the first case, then we have $J_j(\varphi)=J_j(\psi)=0$ which means that $$\rho(s\varphi_j,\varphi_{\sigma_i(s)(j)})\le\sqrt{\delta_0},\quad\rho(s\psi_j,\psi_{\sigma_i(s)(j)})\le\sqrt{\delta_0},\quad\quad\forall s\in F_0.$$ This implies that $$D(s\varphi_j(e),\varphi_{\sigma_i(s)(j)}(e))\le\sqrt{\delta_0},\quad D(s\psi_j(e),\psi_{\sigma_i(s)(j)}(e))\le\sqrt{\delta_0},\quad\quad\forall s\in F_0.$$ Since $s\varphi_j(e)=\varphi_j(s)$ and $s\psi_j(e)=\psi_j(s)$, we have $$D(\varphi_j(s),\varphi_{\sigma_i(s)(j)}(e))\le\sqrt{\delta_0},\quad\quad D(\psi_j(s),\psi_{\sigma_i(s)(j)}(e))\le\sqrt{\delta_0},\quad\quad\forall s\in F_0.$$ Hence, by noting that $D(\varphi_l(e),\psi_l(e))<\epsilon$ for all $l\in[d_i]$, we deduce that $$D(\varphi_j(s),\psi_j(s))\le2\sqrt{\delta_0}+\epsilon<2\epsilon,\quad\quad\forall s\in F_0.$$ Thus, by the choice of $F_0$ we obtain that $$\rho(\varphi_j,\psi_j)<3\epsilon.$$ If any $j\in[d_i]$ encounters the second case, then it follows directly from $J_j(\varphi)=J_j(\psi)>0$ that $(f_1(\varphi_j(s)))_{s\in F_0}=(f_1(\psi_j(s)))_{s\in F_0}$. This implies that $D(\varphi_j(s),\psi_j(s))<\epsilon$, for all $s\in F_0$. By the choice of $F_0$ we also have in this case that $\rho(\varphi_j,\psi_j)<3\epsilon$. We therefore conclude that $\rho_\infty(\varphi,\psi)<3\epsilon$. This proves the claim.

Here the point is to generate a cone from the image of a space $K$ under some $\epsilon$-embedding mapping (\textit{not} from the space $K$ itself). As follows is a slightly more detailed explanation in connection with the construction of the mapping $H_i$. We take an arbitrary $\phi\in\Map(\rho,F_0,\delta_0,\sigma_i)$. Intuitively speaking, for a given $j\in[d_i]$ we say that $\phi_j$ is \textit{good} if $J_j(\phi)=0$, because in this case the family $\{\phi_l(e):l\in[d_i]\}$ carries the needed information fairly well about $\phi_j$ (i.e. it has been already with nearly all the information sufficiently close to $\{\phi_j(s):s\in F_0\}$); while we say that $\phi_j$ is \textit{bad} if $J_j(\phi)>0$, and in this case, the information carried by the family $\{\phi_l(e):l\in[d_i]\}$ is then not able to recover $\phi_j$ (i.e. is not so close to $\phi_j$ as $\{\phi_j(s):s\in F_0\}$) unless we additionally record almost the whole $\phi_j$ (i.e. simply choose to pick out $\{\phi_j(s):s\in F_0\}$ entirely so as to admit some $3\epsilon$-embedding). When we look at each fixed $j\in[d_i]$, $\phi_j$ is always possible to lie in any of these two regions as $\phi\in\Map(\rho,F_0,\delta_0,\sigma_i)$ changes. However, for any fixed $j\in[d_i]$ we \textit{cannot} expect $\phi_j$ to \textit{move in a continuous way} between good and bad regions as $\phi$ moves continuously within $\Map(\rho,F_0,\delta_0,\sigma_i)$. In relation to $\{\phi_l(e):l\in[d_i]\}$ (which gives information properly only to the good region), it is considered suitable for the size of dimension. But the rectangle $\{\phi_k(s):s\in F_0,k\in[d_i]\}$ (that works on both of those two regions) turns out to be considerably large (namely, the dimension of the space $A_1^{|F_0|\cdot d_i}$ is far from satisfactory). The reason why we need build \textit{the} cone in the construction of $H_i$ is that this can record $\phi_j$ (for every $j\in[d_i]$) continuously (with respect to $\phi\in\Map(\rho,F_0,\delta_0,\sigma_i)$) in a more flexible and productive approach.

Now by this claim we get an upper bound for $\Widim_{3\epsilon}(\Map(\rho,F_0,\delta_0,\sigma_i),\rho_\infty)$: $$\Widim_{3\epsilon}(\Map(\rho,F_0,\delta_0,\sigma_i),\rho_\infty)\le\dim(H_i(\Map(\rho,F_0,\delta_0,\sigma_i))).$$ To estimate the term $\dim(H_i(\Map(\rho,F_0,\delta_0,\sigma_i)))$ from above, we need deal with a subset of $[d_i]$, for any $\phi=(\phi_j)_{j\in[d_i]}$ coming from $\Map(\rho,F_0,\delta_0,\sigma_i)$: $$\Omega_\phi(\rho,F_0,\delta_0,\sigma_i)=[d_i]\setminus\{j\in[d_i]:\rho(s\phi_j,\phi_{\sigma_i(s)(j)})\le\sqrt{\delta_0},\;\,\forall s\in F_0\}.$$ For any $\phi=(\phi_j)_{j\in[d_i]}\in\Map(\rho,F_0,\delta_0,\sigma_i)$ we have $$\sum_{j\in[d_i]}\rho(s\phi_j,\phi_{\sigma_i(s)(j)})^2\le\delta_0^2\cdot d_i,\quad\forall\,s\in F_0.$$ It follows that $$\delta_0\cdot|\Omega_\phi(\rho,F_0,\delta_0,\sigma_i)|\le\sum_{j\in[d_i]}\sum_{s\in F_0}\rho(s\phi_j,\phi_{\sigma_i(s)(j)})^2\le|F_0|\cdot\delta_0^2\cdot d_i$$ and hence $$|\Omega_\phi(\rho,F_0,\delta_0,\sigma_i)|\le|F_0|\cdot\delta_0\cdot d_i.$$ This implies that for any given $\phi\in\Map(\rho,F_0,\delta_0,\sigma_i)$ the image $H_i(\phi)\in A_{d_i}\times(CA_1)^{|F_0|\cdot d_i}$ has entries at least $|F_0|\cdot(d_i-|F_0|\cdot\delta_0\cdot d_i)$, which have to take the value $\ast$. More precisely, $H_i(\Map(\rho,F_0,\delta_0,\sigma_i))$ is contained in $$\bigcup_{l\in\mathbb{Z},\;\,0\le l\le|F_0|\cdot\delta_0\cdot d_i}\,\;\bigcup_{k_1,\dots,k_{d_i}\in\{0,1\},\;\,k_1+\cdots+k_{d_i}=l}\,\;A_{d_i}\times\prod_{j=1}^{d_i}(\{\ast\}^{|F_0|})^{1-k_j}\times((CA_1)^{|F_0|})^{k_j},$$ where we set $$\{\ast\}^{|F_0|}\times((CA_1)^{|F_0|})^0=\{\ast\}^{|F_0|},\quad\quad(\{\ast\}^{|F_0|})^0\times(CA_1)^{|F_0|}=(CA_1)^{|F_0|}.$$ Thus, we deduce that $$\dim(H_i(\Map(\rho,F_0,\delta_0,\sigma_i)))\le\dim(A_{d_i})+|F_0|^2\cdot\delta_0\cdot d_i\cdot\dim(CA_1).$$ It follows that
\begin{align*}
&\limsup_{i\to\infty}\frac{\Widim_{3\epsilon}\left(\Map(\rho,F_0,\delta_0,\sigma_i),\rho_\infty\right)}{d_i}\\
&\quad\quad\le\limsup_{i\to\infty}\frac{\dim(H_i(\Map(\rho,F_0,\delta_0,\sigma_i)))}{d_i}\\
&\quad\quad\le\limsup_{i\to\infty}\frac{\dim(A_{d_i})+|F_0|^2\cdot\delta_0\cdot d_i\cdot\dim(CA_1)}{d_i}\\
&\quad\quad\le\limsup_{i\to\infty}\frac{\Widim_\epsilon(K^{d_i},D_\infty)}{d_i}+|F_0|^2\cdot\delta_0\cdot(1+\Widim_\epsilon(K,D))\\
&\quad\quad\le\limsup_{i\to\infty}\frac{\Widim_\epsilon(K^{d_i},D_\infty)}{d_i}+\eta.
\end{align*}
Since $d_i\to\infty$ as $i\to\infty$, we conclude that $$\limsup_{i\to\infty}\frac{\Widim_{3\epsilon}\left(\Map(\rho,F_0,\delta_0,\sigma_i),\rho_\infty\right)}{d_i}\le\eta+\lim_{n\to\infty}\frac{\Widim_\epsilon(K^n,D_\infty)}{n}.$$ This is as desired.
\end{proof}

\medskip

\section*{Appendix A}
The definition of \textit{sofic mean dimension} was originally introduced by Li \cite{Li} (using open covers). However, we notice that the definition which we stated in the body of this paper (via $\epsilon$-width dimension) differs from Li's. This appendix is devoted to clarifying that these two definitions for sofic mean dimension are equivalent to each other.

Let $Y$ be a compact space and $\mathcal{U}$ a finite open cover of $Y$. We denote $$\mathrm{ord}(\mathcal{U})=\max_{y\in Y}\sum_{U\in\mathcal{U}}1_U(y)-1,\quad\quad\mathcal{D}(\mathcal{U})=\min_{\mathcal{V}\succ\mathcal{U}}\mathrm{ord}(\mathcal{V}),$$ where $1_U(y)$ takes the value $1$ (resp. $0$) if $y\in U$ (resp. $y\notin U$) and where $\mathcal{V}\succ\mathcal{U}$ means that $\mathcal{V}$ ranges over finite open covers of $Y$ refining $\mathcal{U}$ (i.e. every element of $\mathcal{V}$ is contained in some element of $\mathcal{U}$). For $d\in\mathbb{N}$ we denote by $\mathcal{U}^d$ the finite open cover of $Y^d$ consisting of all those open sets in $Y^d$ of the form $U_1\times\cdots\times U_d$ (where $U_1,\dots,U_d\in\mathcal{U}$).

Let a sofic group $G$ act continuously on a compact metrizable space $X$. Let $\Sigma=\{\sigma_i:G\to\Sym(d_i)\}_{i\in\mathbb{N}}$ be a sofic approximation sequence for $G$. We fix a compatible metric $\rho$ on $X$. For a finite open cover $\mathcal{U}$ of $X$, a nonempty finite subset $F$ of $G$, $\delta>0$, and $i\in\mathbb{N}$ we denote $\mathcal{D}(\mathcal{U}^{d_i}\cap\Map(\rho,F,\delta,\sigma_i))$ by $\mathcal{D}(\mathcal{U},\rho,F,\delta,\sigma_i)$. Recall that we have already given the definition of $\mdim_\Sigma(X,G)$ in Section 2. Now what we want to see is $$\mdim_\Sigma(X,G)=\sup_{\mathcal{U}}\inf_{F\subset G\text{ finite, }\,\delta>0}\limsup_{i\to\infty}\frac{\mathcal{D}(\mathcal{U},\rho,F,\delta,\sigma_i)}{d_i},$$ where $\mathcal{U}$ runs over finite open covers of $X$.

In fact, this can be shown with a simple and standard argument. The key point is to consider a Lebesgue number of a finite open cover $\mathcal{U}$ with respect to $\rho$, and to use the following fact \cite{LW}: For a finite open cover $\mathcal{U}$ of a compact metrizable space $X$, $\mathcal{D}(\mathcal{U})\le n$ (where $n$ is a nonnegative integer) exactly when there is a compact metrizable space $P$ of topological dimension $n$, which admits a continuous mapping $f:X\to P$ satisfying that there exists a finite open cover $\mathcal{V}$ of $P$ such that $f^{-1}(\mathcal{V})$ refines $\mathcal{U}$. Lastly, we note that although a compatible metric $\rho$ on the space $X$ is also involved in the right-hand side of the above equality, it can be shown that the value defined by the right-hand side of this equality is independent of the choice of compatible metrics $\rho$ on $X$ (for details please refer to \cite[Section 2]{Li}).

\medskip

\section*{Appendix B}
This appendix contains a dichotomy in dimension theory, which leads immediately to an alternative to the expression for sofic mean dimension of double finite actions and full shifts over finite-dimensional alphabets (that is more explicit in some sense, depending upon the reader's preference). For details please refer to \cite{Dranishnikov,Tsukamoto}.

Recall that $\dim(K)$ denotes the Lebesgue covering dimension of a topological space $K$ (with the convention that defines it to be $-\infty$ provided the space is empty). For an arbitrary finite-dimensional compact metrizable space $K$ it was classically known that $2\cdot\dim(K)-1\le\dim(K\times K)\le2\cdot\dim(K)$. Since $\dim(K)$ must be a nonnegative integer, we have either $\dim(K\times K)=2\cdot\dim(K)$ or $\dim(K\times K)=2\cdot\dim(K)-1$. Further, for every $n\in\mathbb{N}$$$\dim(K^n)=\begin{cases}n\cdot\dim(K),&\text{if $K$ satisfies $\dim(K\times K)=2\cdot\dim(K)$}\\n\cdot\dim(K)-n+1,&\text{otherwise}\end{cases}.$$ As a direct consequence, we have $$\inf_{n\in\mathbb{N}}\frac{\dim(K^n)}{n}=\begin{cases}\dim(K),&\text{if $K$ satisfies $\dim(K\times K)=2\cdot\dim(K)$}\\\dim(K)-1,&\text{otherwise}\end{cases}$$ for any finite-dimensional compact metrizable space $K$. This term appears in the quantitative main results of our present paper.

We remark that (by definition) if a finite group $G$ acts continuously on a compact metrizable space $X$ then its mean dimension is equal to $\dim(X)/|G|$. Hence, from the above alternative we get a corollary:
\begin{itemize}\item[]
If a finite group $G$ acts continuously on a finite-dimensional compact metrizable space $X$ and if $\Sigma$ is a sofic approximation sequence for $G$, then $\mdim_\Sigma(X,G)$ agrees with the mean dimension of $(X,G)$ exactly when $\dim(X\times X)=2\cdot\dim(X)$.
\end{itemize}
Moreover, as mentioned in the introduction, the relation between sofic mean dimension and mean dimension was shown by Li \cite[Section 3]{Li} as follows: If an infinite amenable group $G$ acts continuously on a compact metrizable space $X$ and if $\Sigma$ is a sofic approximation sequence for $G$, then $\mdim_\Sigma(X,G)$ coincides with the mean dimension of $(X,G)$. So, combining all these facts and results we see that the mean dimension of the full shift $(K^G,\sigma_G)$, where the alphabet $K$ is finite-dimensional and where the group $G$ is amenable, is equal to $\dim(K)$ (exactly when $K$ satisfies $\dim(K\times K)=2\cdot\dim(K)$), $\dim(K)-1$ (exactly when $G$ is infinite and $K$ satisfies $\dim(K\times K)=2\cdot\dim(K)-1$), or $\dim(K)-1+1/|G|$ (exactly when $G$ is finite and $K$ satisfies $\dim(K\times K)=2\cdot\dim(K)-1$).

\medskip

\end{document}